\documentclass[10pt,reqno,a4paper]{article}
\usepackage[a4paper, total={6in, 8in}]{geometry}
\usepackage{geometry}
\usepackage{amsmath,amsthm,amstext,amscd,amssymb,euscript,url}
\usepackage{mathrsfs}
\usepackage{calrsfs}
\usepackage{epsfig}
\usepackage[inline]{enumitem}
\usepackage{todonotes}
\usepackage{mathtools}
\usepackage[absolute,overlay]{textpos}
\usepackage{graphicx}
\usepackage{subfig}
\usepackage{tikz}
\usetikzlibrary{matrix}
\usetikzlibrary{arrows,positioning}
\usepackage{xcolor}
\usepackage[normalem]{ulem}
\newcommand\JZout{\bgroup\markoverwith
{\textcolor{red}{\rule[.5ex]{2pt}{0.4pt}}}\ULon}
\newcommand{\Id}{\mathrm{Id}}
\usetikzlibrary{3d}
\usepackage{xifthen}
\usetikzlibrary{decorations.pathmorphing}

\renewcommand{\P}{\mathbb{P}}

\newcommand{\Z}{\mathbb Z}
\newcommand{\R}{\mathbb R}
\newcommand{\N}{\mathbb N}

\newcommand{\E}{\mathbb E}
\newcommand{\Zd}{\mathbb Z^d}
\newcommand{\Znd}{\mathbb Z_n^d}

\newcommand{\Rd}{\mathbb R^d}

\renewcommand{\phi}{\varphi}

\newcommand{\loc}{\mathcal{L}}

\newcommand{\Cfrak}{{\mathfrak C}}

\def\1{{\mathchoice {\rm 1\mskip-4mu l} {\rm 1\mskip-4mu l}
{\rm 1\mskip-4.5mu l} {\rm 1\mskip-5mu l}}}

\newtheorem{theorem}{{\small T}{\scriptsize HEOREM}}[section]
\newtheorem*{theorem*}{{\small T}{\scriptsize HEOREM}}
\newtheorem{corollary}{{\bf{\small C}{\scriptsize OROLLARY}}}[section]
\newtheorem{proposition}{{\bf{\small P}{\scriptsize ROPOSITION}}}[section]
\newtheorem*{proposition*}{{\bf{\small P}{\scriptsize ROPOSITION}}}
\newtheorem{lemma}{{\bf{\small L}{\scriptsize EMMA}}}[section]
\newtheorem{remark}{{\bf{\small R}{\scriptsize EMARK}}}[section]
\newtheorem{definition}{{\bf{\small D}{\scriptsize EFINITION}}}[section]
\newtheorem{induction}{{\bf{\small I}{\scriptsize NDUCTIVE HYPOTHESIS}}}[section]
\newtheorem{assumption}{{\bf{\small A}{\scriptsize SSUMPTION}}}[section]

\renewenvironment{proof}[1]
{\noindent{{\bf{\small{ P}{\scriptsize ROOF}}}.}\hspace{0.1cm} #1} {$\;\qed$\newline}

\newenvironment{proofof}[1]{%
    \par\noindent\textit{Proof of #1:}\quad%
}{\hfill$\square$\par}

\newcommand{\beq}{\begin{eqnarray}}
\newcommand{\eeq}{\end{eqnarray}}

\newcommand{\ba}{\begin{align*}}
\newcommand{\ea}{\end{align*}}

\newcommand{\be}{\begin{equation}}
\newcommand{\ee}{\end{equation}}

\newcommand{\bl}{\begin{lemma}}
\newcommand{\el}{\end{lemma}}

\newcommand{\br}{\begin{remark}}
\newcommand{\er}{\end{remark}}

\newcommand{\bt}{\begin{theorem}}
\newcommand{\et}{\end{theorem}}

\newcommand{\bd}{\begin{definition}}
\newcommand{\ed}{\end{definition}}

\newcommand{\bind}{\begin{induction}}
\newcommand{\eind}{\end{induction}}

\newcommand{\basu}{\begin{assumption}}
\newcommand{\esu}{\end{assumption}}

\newcommand{\bp}{\begin{proposition}}
\newcommand{\ep}{\end{proposition}}

\newcommand{\bc}{\begin{corollary}}
\newcommand{\ec}{\end{corollary}}

\newcommand{\bpr}{\begin{proof}}
\newcommand{\epr}{\end{proof}}

\newcommand{\bi}{\begin{itemize}}
\newcommand{\ei}{\end{itemize}}

\newcommand{\ben}{\begin{enumerate}}
\newcommand{\een}{\end{enumerate}}

\newcommand{\caE}{{\mathrsfs E}}

\newcommand{\caF}{{\mathcal F}}

\newcommand{\caI}{{\mathcal I}}

\newcommand{\caM}{{\mathcal M}}

\newcommand{\IRW}{\text{\normalfont IRW}}
\newcommand{\SIP}{\text{\normalfont SIP}}

\newcommand{\bix}{{\mathbf x}}

\renewcommand{\(}{\left(}
\renewcommand{\)}{\right)}
\newcommand{\nn}{\nonumber}

\newmuskip\pFqmuskip

\newcommand\pFq[6][8]{%
	\begingroup 
	\pFqmuskip=#1mu\relax
	\mathcode`\,=\string"8000
	\begingroup\lccode`\~=`\,
	\lowercase{\endgroup\let~}\pFqcomma
	{}_{#2}F_{#3}{\left[\genfrac..{0pt}{}{#4}{#5};#6\right]}%
	\endgroup
}
\newcommand{\pFqcomma}{\mskip\pFqmuskip}

\newcommand*{\myprime}{^{\prime}\mkern-1.2mu}

\newcommand{\norm}[1]{\left\lVert#1\right\rVert}

\newcommand{\xangle}{11}
\newcommand{\yangle}{133}
\newcommand{\zangle}{270}

\newcommand{\xlength}{1}
\newcommand{\ylength}{1}
\newcommand{\zlength}{1}

\pgfmathsetmacro{\xx}{\xlength*cos(\xangle)}
\pgfmathsetmacro{\xy}{\xlength*sin(\xangle)}
\pgfmathsetmacro{\yx}{\ylength*cos(\yangle)}
\pgfmathsetmacro{\yy}{\ylength*sin(\yangle)}
\pgfmathsetmacro{\zx}{\zlength*cos(\zangle)}
\pgfmathsetmacro{\zy}{\zlength*sin(\zangle)}

\begin{document}

\title{{\bf Hydrodynamic limits and non-equilibrium fluctuations for the Symmetric Inclusion Process with long jumps}}
\author{Mario Ayala, and Johannes Zimmer \\
\small{School of Computation, Information and Technology}\\
\small{Chair for Analysis and Modelling}\\
\small{Technische Universität München}\\
{\small Boltzmannstraße 3}\\
{\small 85747 Garching }
\\
\small{Germany}
}
\maketitle

\begin{abstract}
  We consider a $d$-dimensional symmetric inclusion process (\SIP), where particles are allowed to jump
  arbitrarily far apart. We establish both the hydrodynamic limit and non-equilibrium fluctuations for the
  empirical measure of particles. With the help of self-duality and Mosco convergence of Dirichlet forms, we
  extend structural parallels between exclusion and inclusion dynamics from the short-range scenario to the
  long-range setting. The hydrodynamic equation for the symmetric inclusion process turns out to be of
  non-local type. At the level of fluctuations from the hydrodynamic limit, we demonstrate that the density fluctuation field
  converges to a time-dependent generalized Ornstein–Uhlenbeck process whose characteristics are again
  non-local.
\end{abstract}

\section{Introduction}
Within the field of non-equilibrium statistical mechanics, Interacting Particle Systems (IPS) are simple
models to study interactions in complex systems, both in equilibrium and nonequilibrium. Suitable IPS offer
the
possibility to rigorously formalize the micro-to-macro transition~\cite{kipnis1998scaling,demasi2006mathematical}. 

Here we focus on establishing scaling limits for an IPS that can be viewed as the bosonic counterpart of the
well-studied exclusion
process~\cite{spitzer1970interaction,kipnis1998scaling,ferrari1988non,gonccalves2018density,erhard2020non}:
the symmetric inclusion process, introduced in~\cite{giardina2007duality}. Specifically, we aim to derive the
hydrodynamic limit and investigate non-equilibrium fluctuations in the case where particles can interact with
other particles arbitrarily far apart, thus allowing for long jumps.

In the short-range context, the prototypical hydrodynamic equation for both exclusion and inclusion process is
of diffusive type, i.e., it is known that the empirical measure converges in distribution to a measure which
is absolutely continuous and whose density solves the partial differential equation
\begin{equation}
  \label{eq:intro-short}
  \partial_t \rho(t,x) = \alpha \Delta \rho(t,x), 
\end{equation}
with $\alpha > 0$ and suitable initial conditions. Moreover, fluctuations from the hydrodynamic limit are also
well established for the short-range setting. It is known~\cite{ferrari1988non,ayala2021higher} that the
density fluctuation field converges, in equilibrium, both for the exclusion and inclusion process, to a
limiting field that formally can be understood as solving the stochastic partial differential equation (SPDE)
\[
d Y_t = \alpha \Delta Y_t + \sqrt{\rho (\alpha + \sigma \rho) \nabla^2 } d W_t,
\]
with $d W_t$ being space-time white noise. Here the case $\alpha \in \N$ and $\sigma=-1$ corresponds to the
\mbox{$\alpha$-exclusion} process and $\alpha \in \R_+$ with $\sigma=1$ to inclusion dynamics.

Hydrodynamic limits in the context of long-jumps for the exclusion process were first established by
Jara~\cite{jara2008hydrodynamic}, who formulated a Cauchy problem associated with an infinitesimal generator
of non-local type. The first main result of this paper establishes that structural similarities between the
exclusion and inclusion processes persist even in the long-jump setting. Specifically, we show that the
hydrodynamic equation of the long-jump setting coincides with that of the exclusion process, namely
\[
\partial_t \rho(t,x) = \alpha ( \, K_q-\Id ) \, \rho(t,x),
\]
where $\Id$ denotes the identity operator and
\[
K_q f(x) :=q \star f(x)
\]
is the convolution with $q$.

Subsequently, Gon\c{c}alves and Jara~\cite{gonccalves2018density} investigated fluctuations around the
hydrodynamic limit for the exclusion process in equilibrium, considering initialization from Bernoulli
product measures. They demonstrate that the scaling limit of density fluctuations corresponds to a fractional
generalized Ornstein–Uhlenbeck process. Our main contribution lies in the establishment of
\emph{non-equilibrium fluctuations} for the inclusion process, i.e., fluctuations around states which are
not necessarily in equilibrium, for the long-range situation. We find that the density fluctuation field
converges to a time-dependent generalized Ornstein–Uhlenbeck process of characteristics
$\lbrace \alpha( \, K_{q} -\Id ), \rho(t,\cdot) \, \Gamma_t^{\rho} \rbrace$, namely
\[
d Y_t = \alpha( \, K_{q} -\Id ) Y_t \, dt + \sqrt{\rho(t,\cdot) \, \Gamma_t^{\rho}} d W_t,
\]
with
\[
\Gamma_t^{\rho} f(x) := \int (\alpha + \rho(t,y)) q(y-x) \left( f(y)-f(x) \right)^2 \, dy, 
\]
and $\rho(t,x)$ solving the non-local hydrodynamic equation. Reduced to equilibrium, this coincides with the
results of~\cite{gonccalves2018density} for the exclusion process.

To facilitate our analysis, we rely on self-duality, a property exhibited by both the inclusion and the
exclusion process. As a consequence of this property, the $k$-point correlation functions in the model obey
closed-form equations. The availability of closed-form correlations greatly facilitates the derivation of
scaling limits, as it allows us to control the growth of moments. In addition to duality, we depart from the
traditional assumption that an invariance principle holds for the underlying dynamics of a single particle;
instead we consider the convergence of Dirichlet forms in the sense of Mosco. We assume the Mosco convergence
of Dirichlet forms for only a single particle, and from the convergence of $k$ copies ($k \in \N$) of the
independent motions we derive (using results from~\cite{ayala2024mosco}) the convergence of Dirichlet forms
for the inclusion process. The Mosco convergence of $k$ \SIP-particles enables us, in particular, to
express the quadratic variation of the density fluctuation field in terms of the solution of the hydrodynamic
equation, and hence by standard arguments establish the non-equilibrium fluctuations. To the best of our
knowledge, this is the first time that this approach has been employed to establish non-equilibrium
fluctuations of IPS.

The rest of this work is organized as follows. In Section~\ref{Preliminaries}, we provide the necessary
background, including the precise formulation of the model and the concept of
self-duality. Section~\ref{MainResSect} presents our main results along with the required
assumptions. Subsequently, in Section~\ref{GOUPRig}, we rigorously introduce and establish the uniqueness of
what we term a time-dependent generalized Ornstein–Uhlenbeck process of characteristics
$\lbrace (\alpha , K_{q} -Id ), \rho(t,\cdot) , \Gamma_t^{\rho} \rbrace$. The proofs of our main results are
presented in Section~\ref{Proofs}, which is divided into two subsections: one dedicated to the proof of the
hydrodynamic limits and the other to the fluctuations from the hydrodynamic limit. Finally, in the Appendix,
we include auxiliary notions and results, such as the concept of Mosco convergence of Dirichlet forms, and the
proof of the uniqueness of the time-dependent generalized Ornstein–Uhlenbeck process.

\section{Preliminaries}
\label{Preliminaries}

\subsection{The model}
The \emph{Symmetric Inclusion Process} with parameter $\alpha \in \mathbb{R}_+$, abbreviated
$\text{SIP}(\alpha)$, is an interacting particle system where particles randomly hop on a lattice
$\Zd_n:=\tfrac{1}{n}\mathbb{Z}^d$ for $n \in \N$. Configurations of particles are denoted by Greek letters
$\eta$ and $\xi$, belonging to the set $\Omega_n:= \mathbb{N}^{\Zd_n}$, where we reserve the letter $\xi$ for
configurations with a finite number of particles. Here, the expression $\eta(x)$ represents the number of
particles at the site $x \in \Zd_n$. We are interested in the time evolution of the processes
$\{\,\eta_{r(n)t} \mid t \geq 0\, \}$, where $r(n):=n^\beta$. This process can be described in terms of the
infinitesimal generator
\begin{align}\label{SIPgen}
  \loc_n f(\eta) &= n^\beta \sum_{x, y \in \Znd} q(n(y-x)) \, \eta(x) \left( \alpha + \eta(y) \right)
                   \left(  f(\eta-\delta_x + \delta_y) - f(\eta) \right),
\end{align}
for some $\beta \in (0,2)$. Here the subscript $n$ denotes scaling in space (through $\Znd$) and in time
(through $n^\beta t$), and $q \colon \Rd \to [0,\infty)$ is a symmetric, continuous strictly positive function
with total mass equal to one,
\begin{equation}\label{massoneforq} \int_{\Rd} q(x) \, dx=1, \end{equation}
and such that for all $z \in \Rd$
\begin{equation}\label{RescTransProb}
   q(z)= n^{-(d+\beta)}q(z/n),
\end{equation}
for all $n \in \N$. 

\begin{assumption}
  \label{AssumpNashIneq}
  Assume that there exist $C_1, C_2 >0$ such that
  \begin{equation}
    \frac{C_1}{|x-y|^{d+\beta}} \leq q(x-y) \leq \frac{C_2}{|x-y|^{d+\beta}},
  \end{equation} 
  for all $x, y \in \Rd$ with $x \neq y$.
\end{assumption}
Consider the case $n=1$, i.e., ignoring the time-space rescaling. Then the process governed by the
generator~\eqref{SIPgen} can be interpreted as a system of particles equipped with two Poisson clocks in the
following way. Firstly, the first clock rings at random exponential times of rate $\alpha$. When this first
clock rings for a particle at a position $x \in \mathbb{Z}^d$, the particle independently moves to a new
position $y \in \mathbb{Z}^d$ with a probability of $q(y-x)$, unaffected by other particles.  The second clock
rings at a constant rate of one. The particle transitions in this case from position $x$ to $y$ at a rate
proportional to $\eta(y) q(y-x)$.  The process is called inclusion process as a particle at position $x$
  is attracted by the $\eta(y)$ particles position $y$.


\subsection{Self-duality}
The process $\SIP(\alpha)$ satisfies a self-duality property that will facilitate the derivation of the main
results of this work. The notion of self-duality is analogous to that of integrable systems in the sense that
IPS satisfying a moment duality are systems for which the BBGKY hierarchy closes, and as a consequence of this,
the \mbox{$k$-particle} correlation functions obey closed-form
equations. 
Here we introduce this notion in terms of semigroups and generators. Moreover, we explain the additional point
of view of self-duality in terms of compatible coordinate processes.

\subsubsection{Self-duality in terms of semigroups}
Let us denote the set of configurations with a finite number of particles by $\Omega_n^{f} $,
\be
\Omega_n^{f} := \mathop{\bigcup}_{ k \in \N} \Omega_n^{(k)} ,
\end{equation}
where
\be
\Omega_n^{(k)}  = \Big\{ \xi \in \Omega_n : \|\xi\|:= \sum_{x \in \Zd} \xi(x) = k \Big\}.
\end{equation}
A \emph{self-duality function} is a function $D\colon\Omega_n^{f} \times \Omega_n \to\R$
such that
\begin{equation}\label{dual1}
\E_{\eta_0} \big[D(\xi_0,\eta_t)\big]=\E_{\xi_0} \big[D(\xi_t, \eta_0)\big],
\end{equation}
for all $\xi_0 \in \Omega_n^{f}$ and $\eta_0 \in \Omega_n$. Notice that in the left hand side of \eqref{dual1}
we are computing the conditional expectation of the duality function evaluated at the time-evolved $\eta_t$
given the knowledge of the initial configuration $\eta_0$, while on the right hand side we are evolving the
simpler configuration $\xi_t$ given the information $\xi_0$.

For the $\SIP(\alpha)$ the self-duality functions are of factorized form
\begin{equation}
  \label{eq:dual-d}
  D(\xi,\eta)=\prod_{x\in\Znd} d(\xi(x),\eta(x)),
\end{equation}
where the single-site (triangular) duality functions are given by
\begin{equation}\label{classicalsingleSIP}
d(m,n) := \1_{\{m \leq n\}} \, \frac{n!}{(n-m)!} \frac{\Gamma(\alpha)}{\Gamma(\alpha+m)}.
\end{equation}
Notice that these duality functions are multinomials on the $\eta$ variables with degree equal to the number
of particles in the finite configuration $\xi$. In particular, we have $d(0,n) = 1$ and
$d(1,n) = \frac{n!}{(n-1)!} \cdot \frac {\Gamma(\alpha)}{\Gamma(\alpha+1)} = n \cdot \frac 1 \alpha$, hence
with $n = \eta(x)$
\begin{equation}
\label{eq:D-eta}
    \eta(x)= \alpha \, D(\delta_x,\eta).
\end{equation}

\subsubsection{Self-duality in terms of generators}
The self-duality relation~\eqref{dual1} can also be express at the level of generators as
\begin{equation}\label{GenselfdualityIntro}
\loc_n D(\xi,\cdot)(\eta) =\loc_n^{(k)} D(\cdot, \eta)(\xi),
\end{equation}
again for all $\xi \in \Omega_n^{f}$ and $\eta \in \Omega_n$, where $\loc_n$ denotes the
generator~\eqref{SIPgen} and $\loc_n^{(k)}$ its restriction to $\Omega_n^{(k)}$. Whenever the choice of
$k \in \N$ is clear, we denote by $\{\, \xi(t) \mid t \geq 0 \,\}$ the $\Omega_n^{(k)} $-valued Markov
process, namely the process with generator
\begin{equation}\label{SIPgenk}
\loc_n^{(k)} f(\xi) = n^\beta  \sum_{x,y \in \Znd} q(n(y-x)) \xi(x) ( \alpha +  \xi(y))(f(\xi-\delta_x +\delta_y)-f(\xi))
\end{equation}
acting on functions $f\colon\Omega_n^{(k)} \to \R$.

\subsubsection{Self-duality in coordinate notation}
The finite process $\{\, \xi(t) \mid t \geq 0 \,\}$ can be realized in terms of a compatible
(see~\cite{carinci2019consistent} for details on this notion) coordinate process on $\Z_n^{dk}$, where
particles have fixed labels, namely
\begin{equation}
  X^{(k)}(t)= (X_1(t), \ldots, X_k(t)), \qquad X_i(t)\in \Znd, \quad
  \forall i=1, \ldots, k, \end{equation}
and $X_i(t)$ is the position of the $i$th particle at time $t\ge 0$. More
precisely, for $\mathbf x \in \mathbb Z^{kd}_n$ the compatible configuration
$\xi(\mathbf x) \in \Omega_n^{(k)} $ is given by
\begin{equation} \xi(\bix)=\(\xi(\bix)(y), y\in \Znd\)\qquad \text{with}
\qquad \xi(\mathbf x)(y)=\sum_{i=1}^k \1_{x_i=y}.
\end{equation}
The evolution of the coordinate process can be described in terms of the infinitesimal generator
\begin{equation}\label{gencoord}
L_n^{(k)} f(\mathbf x) = n^\beta \sum_{i=1}^{k} \sum_{r \in \Zd} q(r) \Bigg( \alpha +  \sum_{\substack{j=1\\ j \neq i}}^k \mathbf 1_{x_j=x_i +r/n} \Bigg) \( f(\mathbf x^{i,i+r}) -f(\mathbf x) \),
\end{equation}
where $\mathbf x^{i,i+r}$ denotes $\mathbf x$ after moving the particle in position $x_i$ to position $x_i+r/n \in \Znd$. 

By compatibility, see~\cite{carinci2019consistent}, the self-duality relations can then also be expressed in
terms of the coordinate process generator as
\begin{equation}\label{GenselfdualityIntrocoord}
\loc_n D(\xi,\cdot)(\eta) =\loc_n^{(k)} D(\cdot, \eta)(\xi) = L_n^{(k)} D(\xi(\cdot), \eta)(\mathbf x),
\end{equation}
or, equivalently, in terms of semigroups as
\begin{equation}\label{dual1coord}
\E_\eta \big[D(\xi,\eta_t)\big]=\E_\xi \big[D(\xi_t, \eta)\big] = \E_{\mathbf x} \big[D(\xi(\cdot), \eta)(X^{(k)}(t))\big],
\end{equation}
in both cases for all $\xi \in \Omega_n^{f} , \eta \in \Omega_n$, and $\mathbf x \in \mathbb Z_n^{kd}$.

\begin{remark}
  In this work we only apply self-duality with $k \in \lbrace 1, 2, 3, 4 \rbrace$. Self-duality in terms of
  infinitesimal generators allows us to \emph{close equations} (cf. Lemma~\ref{LemmaDynkin} and
  Lemma~\ref{LemmaDrift}), while self-duality at the level of semigroups permits us to control the
  correlations needed to show the vanishing of quadratic variations for the case of hydrodynamic limits
  (cf. Lemma~\ref{LemmaDynkin}) and the validity of key replacements (cf. Lemma~\ref{ConvQuadraVarFluc}) in
  the case of fluctuations.
\end{remark}

\section{Assumptions and main results}\label{MainResSect}


Let $T>0$. The sequence configuration process $\{\, \eta_{r(n)t} \mid t \in [0,T] \, , n \in \N \}$ has the
state space $\Omega^{(n)}$ corresponding to the rescaled lattice $\Zd_n$. We denote by $\P_n$, respectively
$\E_n$, the probability measure, respectively expectation, induced by a given sequence of initial measures
$\{ \mu_n \mid n \in \N \}$ (converging suitably as in Assumption~\ref{HydroAssump} below) and the sequence of
rescaled processes $\lbrace \eta_{r(n)t} \rbrace$ on the path-space $D([0,T];\Omega^{(n)})$. We are interested
on the derivation of hydrodynamic limits and non-equilibrium fluctuations from the hydrodynamic limit for this
sequence of processes. More precisely, we derive scaling limits for the measure-valued processes
\begin{equation}\label{EmpAveg}
    \pi_n(t) := \frac{1}{n^d} \sum_{x \in \Zd} \eta_{r(n)t}(x/n) \, \delta_{x/n},
\end{equation}
and its central limit theorem analogue, the distribution-valued process
\begin{equation}\label{centereddensity}
Y_t^{(n)}(\eta):= \sqrt{n^d} \left( \pi_n(t) - \E_n \left[ \pi_n(t) \right] \right). 
\end{equation} 

\subsection{Assumptions}
In this section we state the following key assumptions made in addition to
  Assumption~\ref{AssumpNashIneq}. In essence, we assume the initial data is suitably chosen, there is an
  invariance principle for the \emph{one-particle} dynamics (which will allow us to rewrite the action of the
  generator in the empirical measure in terms of the one-particle generator) and we assume that the weak
  solution of the hydrodynamic limit exists and is unique.
  
We start with the class of admissible initial measures for the particle configuration at time zero.

\begin{definition}[Measure associated to a density profile]\label{DefMeasAssoc}
  We say that a sequence of measures $\lbrace \mu_n : n \in \N \rbrace$ is \emph{associated} to a density
  profile $\rho_0 \colon \Rd \to \R_{+}$ if the sequence of random measures
  $\lbrace \pi_n(0): n \in \N \rbrace$ defined in~\eqref{EmpAveg} converges in probability, with respect to
  $\P_n$, to the deterministic measure $\rho_0(u) \, du$.
\end{definition}

\begin{assumption}\label{HydroAssump}
The process $\{\,  \eta_{r(n)t} \, \}_{t \geq 0}$ is such that its initial measure $\nu_{\rho_0}^n$ is associated to a profile $\rho_0 \in L^1(\Rd) \cap L^\infty(\Rd)$ in the sense of Definition~\ref{DefMeasAssoc}, and for all $n \in \N$ 
there exist constants $C_k >0$, independent of $n$, such that 
    \begin{equation}
    \sup_{x_1, \ldots, x_k \in \Znd} \left|\alpha^k \int   D\left(\sum_{j=1}^k \delta_{x_j}, \eta \right) \, \nu_{\rho_0}^n(d\eta) - \prod_{j=1}^k \int  \eta(x_j) \,  \nu_{\rho_0}^n(d\eta) \right| \leq \frac{C_k}{n}
    \end{equation} 
    for all $k \leq 4$.
\end{assumption}

\br Notice that Assumption~\ref{HydroAssump} is natural in the derivation of non-equilibirum fluctuations for
interacting particle systems. See for example~\cite[Section 1.2]{ferrari1988non} or more recently the
assumptions on the initial measures in~\cite[Theorem 2.4 and Theorem 2.5]{erhard2020non}.  \er

We now come to the core assumption, namely the analogue of an invariance principle for the one-particle
  dynamics.  Let us denote by $X^{(n)}(t)$ the Markov process on $\Zd_n$ with evolution described by the
\emph{one-particle} generator
\begin{equation}\label{SIPonePartGenResc} 
L_n^{(1)}f(x/n) = r(n) \sum_{y \in \Zd} q(y-x) (f(y/n) -f(x/n)).
\end{equation} 
For the case of finite range interaction, the underlying
\emph{one-particle} random walker is usually implicitly assumed or shown to satisfy an invariance
principle. Here, we will instead assume that the one-particle dynamics is close in some sense to the dynamics
of the jump process whose dynamics can be described in terms of the infinitesimal generator acting on
$f \in C_c^2(\Rd)$ as
\begin{equation}\label{DefL} 
L f(x) = \int_{\Rd} q(y-x) \left( f(y) -f(x) \right) \, dy =: (K_{q} -\Id )f(x).
\end{equation} 

\begin{assumption}\label{RemarkConseqInvariance} 
  Let $\mu_c$ and $dx$ denote the counting measure on $\Zd_n$ and the Lebesgue measure on $\Rd$. Assume that
  $q$ is such that the following holds.
\begin{enumerate}
    \item The sequence of Dirichlet forms 
      \begin{equation}
        \caE_n(f) = \frac{r(n)}{n^d} \sum_{x, y \in \Zd} q(y-x) (f(y/n) -f(x/n))^2,
      \end{equation} 
      acting on functions $f$ in $H_n:= L^2\left(\tfrac{1}{n}\Zd, \frac{1}{n^d} \mu_c\right)$, converges in
      the sense of Definition \ref{MoscoDef} to the Dirichlet form associated with the infinitesimal
      generator~\eqref{DefL}, namely the form
      \begin{align}\label{DefDirFormMac}
        \caE(f) &= \int_{\Rd} \int_{\Rd} q(y) (f(x+y)-f(x))^2 dy \, dx,
      \end{align} 
      acting on elements of the Hilbert space $H:=L^2(\Rd, dx)$.  Moreover, we assume that this convergence
      holds under the convergence of Hilbert spaces
      \begin{equation}
        H_n \to H, 
      \end{equation} 
      where the operators $\Phi_n$ describing the convergence are 
      given by the restrictions operators
      from the smooth compactly supported test functions to $H_n$, see Subsection~\ref{sec:app-Hilbert}. 
    \item For all $\phi \in C_k^\infty(\Rd)$ we have
    \begin{equation}\label{ApproxL01} 
    \lim_{n \to \infty} \frac{1}{n^d} \sum_{x \in \Zd} \left| L_n^{(1)} \phi(x/n) - L \phi(x/n) \right| = 0
    \end{equation} 
    and 
    \begin{equation}\label{ApproxL1}
    \lim_{n \to \infty} \sup_{x \in \Zd} \left| L_n^{(1)} \phi(x/n) - L\phi(x/n) \right| = 0.
    \end{equation} 
\end{enumerate}
\end{assumption}

Finally, we assume that the Cauchy problem of $L$ with initial condition
\mbox{$\rho_0 \in L^\infty(\Rd) \cap L^2(\Rd)$} has a unique weak solution.

\begin{assumption}\label{AssumpUniqueWeakSolu}
  Let $T>0$, and $L$ be given as in~\eqref{DefL}. Assume that $\rho_0 \in L^\infty(\Rd) \cap L^2(\Rd)$ is such
  that there exists a unique $\rho \colon [0,T] \times \Rd \to \R$ being a weak solution to the Cauchy problem
  \begin{equation}
    \begin{cases}
      \partial_t \rho(t,x)=  \alpha ( \, K_{q} -\Id )\rho(t,x) , \nn \\
      \rho(0, x) = \rho_0(x).
    \end{cases}
  \end{equation}  
\end{assumption}

\br  
This assumption has been shown to hold for a family of functions $q$, including 
\begin{equation}
q(x) = \frac{c}{\norm{x}^{d+\beta}},  \nn 
\end{equation}
for some $c>0$. See for example Section 8 of~\cite{jara2008hydrodynamic}.
\er


\subsection{First main result: Hydrodynamic limit}


We now state the hydrodynamic limit, complementing the short-range interaction case
  of~\cite{ayala-valenzuela_hydrodynamic_2016}, which results in the local hydrodynamic
  equation~\eqref{eq:intro-short}.

\begin{theorem}\label{MainHydro}
  Let $T>0$. Under Assumptions~\ref{HydroAssump}, \ref{RemarkConseqInvariance} and~\ref{AssumpUniqueWeakSolu},
  the sequence of measure-valued processes $\lbrace \pi_n(t): t \in [0,T], n \in \N \rbrace$ converges weakly
  to a deterministic measure which is absolutely continuous with respect to the Lebesgue measure. The density
  $\rho(t,x)$ solves
  \begin{equation}\label{HydroEquationLJ} 
    \begin{cases}
      \partial_t \rho(t,x)= \alpha  (\, K_{q} -\Id ) \rho(t,x)\\
      \rho(0, x) = \rho_0(x),
    \end{cases}
  \end{equation} 
  where $\rho_0$ is the profile mentioned in Assumption~\ref{HydroAssump}.
\end{theorem}


\begin{remark}
  \label{RemInitdata}
  Notice that Assumption~\ref{HydroAssump} implies that for every $\delta >0$, we have
  \begin{equation}\label{HydroAssumpEquation}
    \lim_{n \to \infty} \P_{\mu_n} \left(  \left| \frac{1}{n^d} \sum_{x \in \Zd} \eta_{0}(x/n) \,
        \phi(x/n) - \int_{\Rd} \phi(u) \, \rho_0(u) \, du  \right| > \delta  \right) = 0,
  \end{equation} 
  for every continuous function $\phi \colon\Rd \to \R$ with compact
  support. Equation~\eqref{HydroAssumpEquation} determines the initial condition of the hydrodynamic equation.
\end{remark}

\begin{remark} 
  Non-local hydrodynamic results are already known for 
  the symmetric exclusion process with long jumps~\cite{jara2008hydrodynamic}.
  We establish an analogous 
  result for inclusion interaction, and show how, with the help of self-duality, the results
  of~\cite{jara2008hydrodynamic} can be potentially extended for exclusion dynamics where the maximum
  occupancy per site is determined by $\alpha \in \N$.
\end{remark}

\subsection{Second main result: Non-equilibrium fluctuations}

In the case of $\SIP(\alpha)$ with a symmetric transition probability restricted to nearest neighbors,
starting from a reversible measure with homogeneous density $\rho$, it is known (see for
example~\cite{ayala2021higher}) that the sequence of processes
\mbox{$\{\, Y_t^{(n)}(\cdot,\eta) \mid n \in \N, t \in [0,T] \, \}$} converges in distribution to the
distribution-valued process $\{\, Y_t \mid t \in [0,T] \, \}$ which formally solves the Ornstein–Uhlenbeck
SPDE
\begin{equation}\label{GOUSIPalpha} 
  d Y_t = \alpha \Delta Y_t + \sqrt{\rho (\alpha + \rho) \nabla^2 } d W_t:= \alpha \Delta Y_t
  + \sqrt{\rho (\alpha + \rho)  } \nabla d W_t,
\end{equation} 
where loosely speaking the term $\sqrt{\rho (\alpha + \rho) \nabla^2} d W_t$ has to be interpreted by saying
that the integral
\begin{equation}\label{intnablawhite}  
\int_0^t \sqrt{\rho (\alpha + \rho) \nabla^2} d W_s(\phi)
\end{equation} 
is a continuous martingale of quadratic variation
\begin{equation}\label{intnablawhitequadvar}
\rho (\alpha + \rho) \norm{\nabla \phi(x)}^2 t
\end{equation} 
for all test functions $\phi \in S(\Rd)$, and where $\norm{\cdot}$ denotes the usual $L^2$-norm. This can be
made rigorous in terms of martingale problems in the spirit of Chapter 11 in~\cite{kipnis1998scaling}.

Before stating our main result relative to fluctuations, let us introduce the non-local operator
\begin{equation}\label{GenTimOUSPDE}
\Gamma_t^{\rho} f(x) := \int (\alpha + \rho(t,y)) q(y-x) \left( f(y)-f(x) \right)^2 \, dy, 
\end{equation} 
with $\rho(t,x)$ denoting the unique solution to the hydrodynamic equation \eqref{HydroEquationLJ}.

\begin{theorem}
  \label{FluctTheorem}
  Let $T>0$. Under Assumptions~\ref{HydroAssump}, \ref{RemarkConseqInvariance},
  and~\ref{AssumpUniqueWeakSolu}, the sequence of processes
  $\{\, Y_t^{(n)}(\cdot,\eta) \mid n \in \N, t \in [0,T] \, \}$ converges in distribution, as $n \to \infty$,
  with respect to the $J_1$-Skorohod topology of $D([0,T],S(\Rd))$ to a process
  $\{\, Y_t \mid t \in [0,T] \, \}$, which \emph{formally} is the unique solution of the generalized
  time-dependent Ornstein-Uhlenbeck process of characteristics
  $\lbrace L, \rho(t,\cdot) \, \Gamma_t^{\rho} \rbrace$:
  \begin{equation}\label{TDGOUDK} 
    d Y_t =\alpha ( \, K_{q} -\Id ) Y_t dt + \sqrt{\rho(t,\cdot) \, \Gamma_t^{\rho}} d W_t,
  \end{equation} 
  with $d W_t$ being space-time white noise, and $\rho(t,x)$ solves the hydrodynamic equation
  \eqref{HydroEquationLJ}.
\end{theorem}

Analogous to the short-range situation, and at a formal level, the term $\sqrt{\rho(t,\cdot) \,\Gamma_t^{\rho}} d W_t$ has to be interpreted by saying that the integral
\begin{equation}\label{intGammatwhite}  
\int_0^t \sqrt{\rho(t,\cdot) \,\Gamma_t^{\rho}} d W_s(\phi)
\end{equation} 
is a continuous martingale of quadratic variation
\begin{equation}\label{intGammatwhitequadvar}
 \int_0^t \norm{ \sqrt{\Gamma_t^{\rho}\phi(x)}}^2_{L^2(\Rd, \rho(s,x) \, dx)} \, ds
\end{equation} 
for all test functions $\phi \in S(\Rd)$. It is then our task to make rigorous this definition. We refer the reader to Section \ref{GOUPRig} were this task is completed following ideas from \cite{gonccalves2018density}, \cite{erhard2020non}, and \cite{kipnis1998scaling}.

\begin{remark}
  Notice that in~\eqref{GenTimOUSPDE}, we left out a factor $\rho(t,x)$, which appears instead in
  expression~\eqref{TDGOUDK}. The rationale for this choice is to emphasize the analogous nature
    of~\eqref{TDGOUDK} to the Dean-Kawasaki equation.
\end{remark}

\section{Generalized time-dependent Ornstein–Uhlenbeck process}
\label{GOUPRig}
We now properly define 
a generalized time-dependent Ornstein–Uhlen\-beck process of characteristics
$\lbrace L, \rho(t,\cdot) \, \Gamma_t^{\rho} \rbrace$. We call this process \emph{time-dependent} since the
quadratic characteristic $\Gamma_t^{\rho}$ is allowed to depend on time. Formally, we would like to define
this process as being the solution to the SPDE
\begin{equation}\label{GenTimOUSPDEdef}
  d Y_t = L Y_t +  \sqrt{\rho(t,\cdot) \,\Gamma_t^{\rho}} d W_t,
\end{equation}
where $\Gamma_t^{\rho}$ is given is as in~\eqref{GenTimOUSPDE}.  In order to make sense of this definition, we
will make use of martingale problems along the lines of~\cite{gonccalves2018density} and
\cite{erhard2020non}. 
Since the space $S(\Rd)$ is not necessarily invariant under the action of the generator $L$, some care has to
be taken to make sense of the first term in the right hand side of~\eqref{GenTimOUSPDEdef},
see~\cite{gonccalves2018density}. Namely, $Y_t(L\phi)$ is in principle not defined for a general test function
$\phi \in S(\Rd)$. 
This can be overcome introducing the notion of stationary $S\myprime(\Rd)$-valued processes,
see~\cite[Definition 2.5]{gonccalves2018density}. 

\begin{definition}
  \label{DefAdmi}
  We say that an $S\myprime(\Rd)$-valued process $\{\, Y_t \mid t \in [0,T] \, \}$, defined on some
  probability space $(\Omega,\caF,\P)$, is \emph{admissible} if for any $t \in [0,T]$ the process $Y_t$ is a
  white noise of covariance given by
  \begin{equation} \E \left[ Y_t(\phi)^2 \right] = \norm{\phi}_{L^2(\Rd, m(\rho(t,x)) \,
      dx)}^2,
  \end{equation} 
  where $\rho(t,x)$ is the unique solution to~\eqref{HydroEquationLJ}, and $L^2(\Rd, m(\rho(t,x))dx)$ is the
  space of square integrable functions with respect to the measure
  $m(\rho(t,x))\, dx:= \rho(t,x)(\alpha +\rho(t,x)) \, dx$.
\end{definition}

\begin{remark}
  Notice that if $\{\, Y_t \mid t \in [0,T]  \, \}$ is admissible then the following is true.
  \begin{enumerate}
  \item The mappings $\phi \mapsto Y_t(\phi) \colon S(\Rd) \subset L^2(\Rd, m(\rho(t,x))dx)) \to L^2(\Omega)$
    are uniformly continuous for every $t \in [0,T]$.
  \item The time-indexed family of mappings $\{\, \phi \mapsto Y_t(\phi) \mid t \in [0,T] \, \}$ is
    equicontinuous.
  \end{enumerate}
\end{remark} 
The next result is a consequence of this observation and~\cite[Footnote~5]{gonccalves2018density}.

\begin{lemma}
  \label{LemmaWellDefYLF}
  For any $\epsilon>0$, let $\Phi_\epsilon(x):= e^{-\epsilon^2 x^2}$. Assume $\{\, Y_t \mid t \in [0,T] \, \}$
  is an admissible process. Let $f \colon [0,T] \to S(\Rd)$ be differentiable. Then the process
  $\{\, \caI_t(f) \mid t \in [0,T] \, \}$ given by
  \begin{equation}q
    \caI_t(f) := \lim_{\epsilon \to 0} \int_0^t Y_s( \Phi_\epsilon Lf_s ) \, ds
  \end{equation} 
  is well-defined.
\end{lemma}

\begin{definition}
  An admissible process $Y$ taking values in $C([0,T], S\myprime(\Rd))$ is called a \emph{generalized
    time-dependent Ornstein–Uhlenbeck process of characteristics
    $\lbrace L,\rho(t,\cdot) \, \Gamma_t^{\rho} \rbrace$} if for all $\phi \colon [0,T] \to S(\Rd)$, the
  processes
  \begin{equation}
    M_t(\phi) = Y_t(\phi) -Y_0(\phi)  - \int_0^t Y_s((\partial_s + L) \phi) \, ds
  \end{equation} 
  and
  \begin{align}
    N_t(\phi) &= M_t(\phi)^2 -   \int_0^t \int \int \rho(s,x) \, \left(\alpha +\rho(s,y) \right) \, q(y-x)
                \left( \phi(y)- \phi(x) \right)^2 \, dy  \, dx \, ds
  \end{align} 
  are $\caF_t$-martingales. 
\end{definition}

To conclude this section, the following proposition establishes uniqueness of solutions to the
SPDE~\eqref{GenTimOUSPDEdef}.

\begin{proposition}
  \label{PropoUniquenessGTOUP} 
  Let $\{\, Y_t^{(1)} \mid t \in [0,T] \, \}$ and $\{\, Y_t^{(2)} \mid t \in [0,T] \, \}$ be two generalized
  time-dependent Ornstein–Uhlenbeck process of characteristcs
  $\lbrace L,\rho(t,\cdot) \, \Gamma_t^{\rho} \rbrace$. Then the processes
  $\{\, Y_t^{(1)} \mid t \in [0,T] \, \}$ and $\{\, Y_t^{(2)} \mid t \in [0,T] \, \}$ have the same
  distribution.
\end{proposition}

The proof of this result is standard and follows the lines of~\cite[Section 11.4]{kipnis1998scaling}. However,
some extra care has to be taken due to the fact that the square characteristic
$\lbrace \rho(t,\cdot) \, \Gamma_t^{\rho} \rbrace$ of the generalized time-dependent Ornstein–Uhlenbeck
process depends 
on time. 
For convenience of the reader we include the proof in Appendix~\ref{AppendixUniquenessGTOUP}.

\section{Proof of main results}
\label{Proofs}

To prove Theorem~\ref{MainHydro} and Theorem~\ref{FluctTheorem}, we first show tightness of the distributions
of the sequences. Then, thanks to Assumption~\ref{AssumpUniqueWeakSolu}, and
Proposition~\ref{PropoUniquenessGTOUP}, we establish uniqueness of limit points via characterization with
martingale problems. Since a relatively compact sequence on a metrizable space with only one accumulation
point is necessarily convergent, Theorem~\ref{MainHydro} and Theorem~\ref{FluctTheorem} follow.

\subsection{Proof of the hydrodynamic limit}
\label{ProofHydro}

\subsubsection{Consequences of self-duality: hydrodynamics}
\label{ConsDualHydro}
We first state some straightforward, but useful, consequences of self-duality that will be used to establish
the hydrodynamic limit.

\subsubsection*{Growth control}

\begin{proposition}
  \label{Growthcontrolone}
  Under Assumption~\ref{HydroAssump}, we have
  \begin{equation}
    \label{Growthcontroloneineq} 
    \E_n\left[  \eta_{r(n)t}(x/n) \right] \leq  C_1 + \norm{\rho_0}_\infty,
  \end{equation} 
  for all $t \geq 0$, and all $x \in \Zd$.
\end{proposition}

\begin{proof} 
  We use~\eqref{eq:D-eta}, the definition of $\E_n$ and then duality with the coordinate representation.
\begin{align}
\E_n\left[  \eta_{r(n)t}(x/n) \right]  &= \alpha \E_n\left[  D(\delta_{x/n}, \eta_{r(n)t}) \right]  = \alpha \int_\Omega \E_\eta  \left[  D(\delta_{x/n}, \eta_{r(n)t}) \right] \, \nu_{\rho_0}^n (d \eta) \nn \\
&\hspace{-2.5cm}= \alpha \int_\Omega \E_{x/n}\left[  D(\delta_{X(r(n)t)}, \eta )\right] \, \nu_{\rho_0}^n(d \eta) = \alpha \sum_{x\myprime \in \Zd} p_{r(n)t}^{(1)}(x/n,x\myprime/n) \int_\Omega D(\delta_{x\myprime/n},\eta) \nu_{\rho_0}^n(d \eta) \nn \\
&\hspace{-2.5cm}\leq  \sum_{x\myprime \in \Zd} p_{r(n)t}^{(1)}(x/n,x\myprime/n)  \left| \alpha \int_\Omega D(\delta_{x\myprime/n},\eta)\nu_{\rho_0}^n(d \eta) - \rho_0(x\myprime/n)  \right| + \sum_{x\myprime \in \Zd} p_{r(n)t}^{(1)}(x/n,x\myprime/n) \rho_0(x\myprime/n) \nn \\
&\hspace{-2.5cm}\leq  \frac{C_1}{n} + \sum_{x\myprime \in \Zd} p_{r(n)t}^{(1)}(x/n,x\myprime/n) \rho_0(x\myprime/n), 
\end{align}
where for $k \in \N$, the function $p_{r(n)t}^{(k)}$ denotes the transition probability of $k$-particles; in
the last line we used Assumption \ref{HydroAssump}.  \end{proof}

Similarly, we can also compute two-point correlations:

\begin{proposition}
  \label{Growthcontroltwo}  
  Under Assumption \ref{HydroAssump} we have
  \begin{equation}
    \E_n\left[  \eta_{r(n)t}(x/n) \, \eta_{r(n)t}(y/n) \right] \leq C_{1,2}^* \,  \norm{\rho_0}_\infty^{(2,1)}
    +  \left( \frac{1}{\alpha} C_{1,2}^* \norm{\rho_0}_\infty^{(2,1)}  + C_1 + \norm{\rho_0}_\infty \right) \1_{x=y}
  \end{equation} 
  for all $t \geq 0$, and all $x, y \in \Zd$, and where
  \begin{equation}
    C_{1,2}^* := 5 \max \left\lbrace C_2, C_1^2, C_1, 1  \right\rbrace,
    \quad
    \text{and}
    \quad
    \norm{\rho_0}_\infty^{(2,1)}  := \max \left\lbrace \norm{\rho_0}_\infty^2, \norm{\rho_0}  \right\rbrace.
  \end{equation} 
\end{proposition} 

\begin{proof}
  The proof of this proposition is in the same spirit as the one for one particle, but slightly more
  involved. We include it for the reader's convenience. First notice that for all $x, y \in \Znd$ we have
\begin{equation}
  \eta(x) \, \eta(y) =  \alpha^2 \1_{x \neq y} d(\delta_x,\eta) \, d(\delta_y,\eta)
  + \alpha \1_{x=y} \left( (\alpha+1) d(2\delta_x,\eta) + d(\delta_x,\eta) \right),
\end{equation}
which is a direct consequence of~\eqref{eq:dual-d} and~\eqref{classicalsingleSIP}.  Then we have
\begin{align}\label{twopointscorr}
 \E_n\left[  \eta_{r(n)t}(x/n) \, \eta_{r(n)t}(y/n) \right] &=  \alpha^2 \1_{x \neq y} \, \E_n\left[ D(\delta_{x/n}+\delta_{y/n},\eta_{r(n)t})\right] \nn \\
 &\hspace{-2cm}+ \alpha \1_{x=y} \left( (\alpha+1)  \E_n\left[ D(2\delta_{x/n},\eta_{r(n)t})\right] +  \E_n\left[ D(\delta_{x/n},\eta_{r(n)t})\right] \right) \nn \\ 
 &\hspace{-2cm}=  \alpha^2 \1_{x \neq y}  \sum_{x\myprime, y\myprime \in \Zd} p_{r(n)t}^{(2)}(x/n,y/n; x\myprime/n, y\myprime/n) \int_\Omega D(\delta_{x\myprime/n}+\delta_{y\myprime/n},\eta) \, \nu_{\rho_0}^n(d\eta) \nn \\
 &\hspace{-2cm}+ \alpha (\alpha+1) \1_{x = y}  \sum_{x\myprime, y\myprime \in \Zd} p_{r(n)t}^{(2)}(x/n,x/n; x\myprime/n, y\myprime/n) \int_\Omega D(\delta_{x\myprime/n}+\delta_{y\myprime/n},\eta) \, \nu_{\rho_0}^n(d\eta) \nn \\
 &\hspace{-2cm}+\alpha \1_{x = y}  \sum_{x\myprime \in \Zd} p_{r(n)t}^{(1)}(x/n,x\myprime/n) \int_\Omega D(\delta_{x\myprime/n},\eta) \, \nu_{\rho_0}^n(d\eta).
\end{align}
Moreover, by applying multiple times Assumption~\ref{HydroAssump} on the initial data, we have 
\small
\begin{align}\label{DsecondorderEstimate}
\alpha^2\int_\Omega &D(\delta_{x\myprime/n}+\delta_{y\myprime/n},\eta) \, \nu_{\rho_0}^n(d\eta) \nn \\
&\leq \left| \alpha^2\int_\Omega D(\delta_{x\myprime/n}+\delta_{y\myprime/n},\eta) \, \nu_{\rho_0}^n(d\eta) -\alpha^2\int_\Omega D(\delta_{x\myprime/n},\eta) \, \nu_{\rho_0}^n(d\eta) \int_\Omega D(\delta_{y\myprime/n},\eta) \, \nu_{\rho_0}^n(d\eta)  \right| \nn \\
&+ \left|\alpha^2 \int_\Omega D(\delta_{x\myprime/n},\eta) \, \nu_{\rho_0}^n(d\eta) \int_\Omega D(\delta_{y\myprime/n},\eta) \, \nu_{\rho_0}^n(d\eta)  - \alpha\rho_0(x\myprime/n) \int_\Omega D(\delta_{y\myprime/n},\eta) \, \nu_{\rho_0}^n(d\eta) \right| \nn \\
&+ \left| \alpha \rho_0(x\myprime/n) \int_\Omega D(\delta_{y\myprime/n},\eta) \, \nu_{\rho_0}^n(d\eta) - \rho_0(x\myprime/n) \rho_0(y\myprime/n) \right|+ \rho_0(x\myprime/n) \rho_0(y\myprime/n) \nn \\
&\leq \frac{C_2}{n} + \frac{C_1}{n} \left| \alpha \int_\Omega D(\delta_{y\myprime/n},\eta) \, \nu_{\rho_0}^n(d\eta) \right| + \frac{C_1}{n} \rho_0(x\myprime/n) + \rho_0(x\myprime/n) \rho_0(y\myprime/n) \nn \\ 
&\leq \frac{C_2}{n} + \frac{C_1^2}{n^{2}} + \frac{C_1}{n} \rho_0(x\myprime/n) +  \frac{C_1}{n} \rho_0(y\myprime/n) + \rho_0(x\myprime/n) \rho_0(y\myprime/n).
\end{align} 
\normalsize
\end{proof} 

\begin{remark}
  \label{RemBounExpRates} 
  Notice that in particular 
  \begin{align}
    \label{RemBounExpRateseq} 
    \E_n \left[  \eta_{r(n) t}(x/n) \left( \alpha + \eta_{r(n)t}(y/n) \right) \right]
    &\leq \alpha(C_1 + \norm{\rho_0}_\infty) + C_{1,2}^* \,  \norm{\rho_0}_\infty^{(2,1)} \nn \\
    &+  \left( \frac{1}{\alpha} C_{1,2}^* \norm{\rho_0}_\infty^{(2,1)}  + C_1 + \norm{\rho_0}_\infty \right) \1_{x=y},
\end{align} 
for all $x,y \in \Zd$.
\end{remark}

\subsubsection{Dynkin Martingales: hydrodynamics}
\label{HydroRelevMart}

We will make extensive use of integrals of non-negative test functions $\phi$ in $C_k^\infty(\Rd)$ against the
measures $\{\, \pi_n( \cdot) \mid n \in \N \, \}$. It is then convenient to give a name to these integrals,
\begin{equation}
  \int \phi \, d\pi_n(t) =: \pi_t^{(n)}(\phi,\eta) = \pi_t^{(n)}(\phi). 
\end{equation}
We will abuse notation and write $\pi_t^{(n)}(\phi)$ whenever it is convenient and clear. 
We have the well-known Dynkin formula have at our disposal, which says that the processes
\begin{equation}
  \label{DynkinMartHydro}
  M_t^{(n)}(\phi) := \pi_t^{(n)}(\phi,\eta) -\pi_0^{(n)}(\phi,\eta) -\int_0^t \loc_n \pi_s^{(n)}(\phi,\eta) \, ds
\end{equation} 
are mean zero martingales with quadratic variation
\begin{equation}
\langle M_t^{(n)}(\phi) \rangle = \int_0^t \Gamma_n \pi_s^{(n)}(\phi,\eta) \, ds, 
\end{equation} 
where the operator $\Gamma_n$ is the so-called Carré du Champ operator associated to the generator $\loc_n$.

First, we have the following application of duality with one particle.

\begin{lemma}
  \label{LemmaDynkin}
  The Dynkin martingale~\eqref{DynkinMartHydro} has the following properties.
  \begin{enumerate}
  \item For every $\phi$ and $t \geq 0$, it can be rewritten as
    \begin{equation}
      \label{DynkinMartHydroclosed}
      M_t^{(n)}(\phi) = \pi_t^{(n)}(\phi,\eta) -\pi_0^{(n)}(\phi,\eta) - \int_0^t  \pi_s^{(n)}(L_n^{(1)} \phi,\eta) \, ds,
    \end{equation} 
    where $L_n^{(1)}$ is the one-particle generator of \eqref{SIPonePartGenResc}.
  \item For any $0<T<\infty$ we have
    \begin{equation}
      \label{vanishingquadvariahydro} 
      \lim_{n \to \infty} \E_n \left[ \sup_{t \in [0,T] } M_t^{(n)}(\phi)^2 \right] = 0, 
    \end{equation} 
    for all $\phi \in C_k^\infty(\Rd)$.
  \end{enumerate}
\end{lemma}

\begin{proof} 
  To show~\eqref{DynkinMartHydroclosed}, it is enough to compute the action of the generator $\loc_n$ on
  $\pi_s^{(n)}(\phi,\eta)$ written in terms of self-duality functions. By linearity and self-duality we have
  \begin{align}
    \label{ActionGenOnepart}
    \loc_n \pi_t^{(n)}(\phi,\eta) &= \frac{\alpha}{n^d} \sum_{x \in \Zd} \phi(x/n) \loc_n D(\delta_{x/n}, \eta_{r(n)t}) = \frac{\alpha}{n^d} \sum_{x \in \Zd} \phi(x/n) \loc_n^{(1)} D(\delta_{x/n}, \eta_{r(n)t}) \nn \\
    &= \frac{\alpha}{n^d} \sum_{x \in \Zd} \phi(x/n) L_n^{(1)} D(\delta_{x/n}, \eta_{r(n)t}) = \frac{\alpha}{n^d} \sum_{x \in \Zd} L_n^{(1)}\phi(x/n)  D(\delta_{x/n}, \eta_{r(n)t}) \nn \\
    &= \pi_t^{(n)}(L_n^{(1)} \phi,\eta),
\end{align}
where in the last equality we used the reversibility of the generator $L_n^{(1)}$ with respect to the counting
measure in $\Zd$.

To show~\eqref{vanishingquadvariahydro}, by Doob’s maximal inequality we have
\begin{equation}
  \E_n \left[ \sup_{t \in [0,T] } M_t^{(n)}(\phi)^2 \right] \leq 4 \E_n \left[  M_T^{(n)}(\phi)^2 \right]
  = 4 \E_n \left[ \int_0^T \Gamma_n \pi_s^{(n)}(\phi,\eta) \, ds \right].
\end{equation} 
We then compute
\begin{align}
  \Gamma_n \pi_{\cdot}^{(n)}(\phi,\eta)
  &= r(n) \sum_{x,y} q(x-y) \eta(x/n) (\alpha + \eta(y/n)) \left( \pi_{\cdot}^{(n)}(\phi,\eta-\delta_{x/n} + \delta_{y/n}) -\pi_{\cdot}^{(n)}(\phi,\eta)  \right)^2 \nn \\
 &= \frac{r(n)}{n^{2d}} \sum_{x,y} q(x-y) \eta(x/n) (\alpha + \eta(y/n)) \left( \phi(y/n)-\phi(x/n)  \right)^2.
\end{align}
By~\eqref{RemBounExpRateseq} we have
\begin{align}\label{LastEquaSupM}
  \E_n \left[ \sup_{t \in [0,T] } M_t^{(n)}(\phi)^2 \right]
  &\leq C(T, \rho_0) \frac{r(n)}{n^{2d}} \sum_{x,y} q(x-y) \left( \phi(y/n)-\phi(x/n)  \right)^2 \nn \\
 &=  \frac{C(T, \rho_0)}{n^{3d}} \sum_{x,y} q((x-y)/n) \left( \phi(y/n)-\phi(x/n)  \right)^2,
\end{align}
where in the last line we used~\eqref{RescTransProb}, recalling that $r(n)=n^\beta$. The fact that the sum in
the right-hand side of~\eqref{LastEquaSupM} is of order $O(n^2)$ concludes the proof.
\end{proof}

\subsubsection{Tightness: hydrodynamics}

The processes $\{\, \pi_n( \cdot) \mid n \in \N \, \}= \{\, \pi_t^{(n)} \mid n \in \N, t \geq 0 \, \}$ have
their paths in the c\`adl\`ag space of right-continuous functions with left-sided limits from
$[0,\infty)$ to the space of Radon measures $\caM_+(\Rd)$, denoted by $D([0,\infty),\caM_+(\Rd))$. Let $Q_n$
be the distribution of these processes, defined for Borel sets $B \subset D([0,\infty),\caM_+(\Rd))$ by
\begin{equation}
  Q_n(B) := \P_n (\pi_n(\cdot) \in B),
\end{equation} 
where as before $\P_n$ is the measure induced by $\eta_{r(n)t}$.

\begin{lemma}
  \label{LemmaTightHydro} 
  The sequence of measures $\{\, Q_n \mid n \in \N   \, \}$ is tight in $D([0,\infty),\caM_+(\Rd))$.
\end{lemma}

\begin{proof}
  By Prohorov's theorem (see for example~\cite[Theorem~1.3]{kipnis1998scaling}
  or~\cite[Appendix~A.10]{seppalainen2008translation}, it is enough to show the following:
\begin{enumerate}
\item Compact containment: For each $t \in [0,\infty)$ and any $\epsilon >0$ there exists a compact set
  $K \subset \caM_+(\Rd)$ such that
  \begin{equation}
    \inf_{n} \P_n ( \pi_t^{(n)} \in K) > 1- \epsilon.
  \end{equation} 
\item Modulus of continuity: For every $\epsilon >0$ and any $T \in (0,\infty)$, there exists a $\delta>0$
  such that
  \begin{equation}
    \limsup_{n \to \infty} \P_n\left( \omega(\pi_n(\cdot),\delta,T) \geq \epsilon  \right) \leq \epsilon,
  \end{equation} 
  where the (modified) modulus of continuity is given by
  \begin{equation}
    \omega(\pi,\delta,T) =\sup \lbrace d_{\caM}(\pi(s),\pi(t)) : s, t \in [0,T], |s-t| \leq \delta \rbrace, 
  \end{equation} 
  with
  \begin{equation}
    d_{\caM}(\pi_1, \pi_2) := \sum_{i \geq 1} 2^{-i} \min \left\lbrace \int \phi_i \, d(\pi_1 - \pi_2) , 1 \right\rbrace, 
  \end{equation} 
  where $\phi_i \in C_k^\infty(\Rd)$ is a rich enough countable family of functions (see~\cite[Appendix
  A.10]{seppalainen2008translation} for more details).
\end{enumerate}

To show compact containment, fix $\epsilon >0$, and let us consider the set
\begin{equation}
  K = \left\lbrace \mu \in \caM_+(\Rd) \mid \mu([-k,k]^d)
    \leq \frac{\alpha \left(C_1 + \norm{\rho_0}_\infty\right) }{\epsilon} (2k+1)^d
    \quad \text{for all } [-k,k]^d \subset \Rd  \right\rbrace, 
\end{equation} 
where $\rho_0$ and $C_1$ are already given in
Assumption~\ref{HydroAssump}. By~\cite[Proposition~A.25]{seppalainen2008translation}, $K$ is a pre-compact
subset of $\caM_+(\Rd)$ and hence for $\bar{K}$ being the closure of $K$ in $\caM_+(\Rd)$ with respect to
  $d_{\caM}$
\begin{equation}
\P_n ( \pi_n(t) \in \bar{K})  \geq  \P_n ( \pi_n(t) \in K).
\end{equation} 
Then by the Markov inequality
\begin{align} 
  \P_n \left( \pi_n(t)([-k,k]^d) \right.
  &\geq \left. \alpha (2k+1)^d \epsilon^{-1} \left(C_1 + \norm{\rho_0}_\infty\right) \right) \nn \\
  &\leq \frac{\epsilon}{\alpha (2k+1)^d n^d \left(C_1 + \norm{\rho_0}_\infty\right)}
    \sum_{ \substack{x \in [-k,k]^d \\ x \in n^{-1} \Zd}} \E_n\left[  \eta_{r(n)t}(x/n) \right] \leq \epsilon, \nn 
\end{align}
where in the last line we used~\eqref{Growthcontroloneineq} from Proposition~\ref{Growthcontrolone} and the
fact that there are $(2kn+1)^d$ points in $[k,k]^d \cap n^{-1} \Zd$.

For the modulus of continuity we first establish the bound
\begin{align}
  \label{BoundModulusone}
  \omega(\pi_n(\cdot), \delta, T)
  &= \sup_{\substack{|s-t| \leq \delta \\ s,t \in[0,T]}} \sum_{j=1}^\infty 2^{-j} (1 \wedge |\pi_n(s, \phi_j)- \pi_n(t,\phi_j)|) \leq 2^{-m} + \sup_{\substack{|s-t| \leq \delta \\ s,t \in[0,T]}} \sum_{j=1}^m  |\pi_n(s, \phi_j) -\pi_n(t,\phi_j)|.
\end{align}
Since $m$ is arbitrary, it is enough to show that the second term in the right-hand side
of~\eqref{BoundModulusone} vanishes in expectation. More precisely we show
\begin{equation}
  \lim_{\delta \to 0}
  \limsup_{n \to \infty} \E_n \left[ \sup_{\substack{|s-t| \leq \delta \\ s,t \in[0,T]}} |\pi_n(s, \phi_j)
    -\pi_n(t,\phi_j)|^2 \right] = 0.
\end{equation}  
From Dynkin's formula, and the basic inequality $(a+b)^2 \leq 2 a^2 + 2 b^2$, it follows that 
\begin{align}
\E_n \left[ \sup_{\substack{|s-t| \leq \delta \\ s,t \in[0,T]}}  |\pi_n(s, \phi) -\pi_n(t,\phi)|^2 \right] &\leq  2 \E_n \left[  \sup_{\substack{|s-t| \leq \delta \\ s,t \in[0,T]}} \left| M_t^{(n)}(\phi) - M_s^{(n)}(\phi)  \right|^2 \right] \nn \\
&{}\quad+2 \E_n \left[  \sup_{\substack{|s-t| \leq \delta \\ s,t \in[0,T]}} \left( \int_s^t \loc_n \pi_n(r,\phi) \, dr \right)^2 \right].
\end{align}
By Lemma~\ref{LemmaDynkin} we can see that the first term in the right-hand side vanishes as $n$ goes to
infinity. We continue the analysis of the second term. By~\eqref{ActionGenOnepart} 
\begin{align}
  \E_n \left[  \sup_{\substack{|s-t| \leq \delta \\ s,t \in[0,T]}} \left( \int_s^t \loc_n \pi_n(u,\phi) \, du \right)^2 \right] &= \frac{1}{n^{2d}} \E_n \left[  \sup_{\substack{|s-t| \leq \delta \\ s,t \in[0,T]}} \left( \int_s^t \sum_{x \in \Zd} \eta_{r(n)u}(x/n) L_n^{(1)} \phi(x/n)  \, du \right)^2 \right] \nonumber \\
  &\leq \frac{\delta}{n^{2d}} \E_n \left[  \sup_{\substack{|s-t| \leq \delta \\ s,t \in[0,T]}}  \int_s^t \left( \sum_{x \in \Zd} \eta_{r(n)u}(x/n) L_n{(1)} \phi(x/n) \right)^2 \, du  \right],
\end{align}
where we used Cauchy-Schwarz inequality. We then get rid of the supremum by extending the time integral to the
whole interval $[0,T]$ and use Proposition~\ref{Growthcontroltwo} to obtain
\begin{align}\label{Thightnesslast}
  \E_n \left[  \sup_{\substack{|s-t| \leq \delta \\ s,t \in[0,T]}} \left( \int_s^t \loc_n \pi_n(u,\phi) \, du \right)^2 \right] &\leq  \delta T \alpha^2 C_{1,2}^*\norm{\rho_0}_\infty^2 \frac{1}{n^{2d}}\sum_{x, y \in \Zd}  \left| L_n^{(1)} \phi(x/n) \right| \left|L_n^{(1)} \phi(y/n) \right| \nonumber \\
  &+ \delta T \alpha \left( C_{1,2}^* \norm{\rho_0}_\infty^2   
  + C_1 + \norm{\rho_0}_\infty \right)  \frac{1}{n^{2d}} \sum_{x \in \Zd}  \left( L_n^{(1)}\phi(x/n) \right)^2. 
\end{align}
We conclude since by Assumption~\ref{RemarkConseqInvariance} the double sum on the right-hand side is bounded
and hence vanishes as $\delta \to 0$. Second, due to standard estimates and the factor $\tfrac{1}{n^{2d}}$, as
first $n \to \infty$ and then $\delta \to 0$ the second term vanishes as well.
\end{proof}

\subsubsection{Uniqueness of limit points: hydrodynamics}
Let $\pi$ be a limit point of the sequence $\pi_n$, relabelled for convenience; existence of $\pi$ is ensured by
  Lemma~\ref{LemmaTightHydro}. 
  In order
  to conclude uniqueness, 
  the idea is to use Lemma~\ref{LemmaDynkin} to obtain that 
$\pi_t$ satisfies
\begin{equation}
  \label{spaceweakform} 
  \pi_t(\phi) = \pi_0(\phi) + \int_0^t \pi_s(L\phi) ds,
\end{equation} 
for any $\phi \in C_k^\infty(\Rd)$. However, there is an additional need of justification to establish the
limit
\begin{equation}
  \label{MeanReplaGen}
  \pi_s^{(n)}(L_n^{(1)} \phi) \to \pi_s(L \phi).
\end{equation} 
It is enough to show that
\begin{equation}
  \lim_{n \to \infty} \E_n \left[ \left( \int_0^t \pi_s^{(n)}(L_n^{(1)} \phi)
      - \pi_s(L \phi)  \, ds\right)^2 \right] =0,
\end{equation} 
which along the same lines as in~\eqref{Thightnesslast}, is a consequence of~\eqref{ApproxL01} and
Proposition~\ref{Growthcontroltwo}.

\begin{remark} 
  By considering time-dependent test functions $\phi$, we can obtain the following space-time formulation
  of~\eqref{spaceweakform}:
  \begin{equation}\label{spacetimeweakform} 
    \pi_t(\phi(t,\cdot)) = \pi_0(\phi(0,\cdot)) + \int_0^t \pi_s(( \partial_s+ L  )\phi(s,\cdot)) \,ds.
  \end{equation} 
\end{remark}

\subsubsection*{Propagation of absolute continuity}
At time zero, the limiting measure $\pi_0$ is absolutely continuous, with density $\rho_0$, with respect to
the Lebesgue measure of $\Rd$. Let $S \subset \Rd$ be a measurable set of Lebesgue measure zero. Consider the
function
\begin{equation*}
  \phi(s,x) = S_{(t-s)}^L \1_{S}(x), 
\end{equation*} 
where $S_{(t-s)}^L$ is the semi-group of the operator $L$. Notice that $\phi(t,x)$ is regular enough to
substitute in~\eqref{spacetimeweakform}, and that it also satisfies
\begin{equation*}
  (L + \partial_t) \phi(s,x) = 0 \nn 
\end{equation*} 
and
\begin{equation*}
  \lim_{s \to 0} \phi(s,x) = \1_{S}(x), 
\end{equation*} 
for all $x \in \Rd$. This allows us to conclude the absolute continuity of $\pi_t$, namely
\begin{equation*}
  \pi_t( S) = 0.
\end{equation*}

\subsubsection*{Characterization of limit points}
Now that we have established absolute continuity of the measure $\pi_t$, we can
write~\eqref{spacetimeweakform} in terms of its density $\rho(t,x)$, namely
\begin{equation}
  \int \rho(t,x) \phi(t,x) \, dx = \int \rho(0,x) \phi(0,x) \, dx
  + \int_0^t \int \rho(s,x) (\partial_s + L ) \phi(s,x) \, dx. \nn 
\end{equation} 
Thus, we have characterized a generic limit point $\pi_t$ as concentrated on weak solutions
of~\eqref{HydroEquationLJ}. The uniqueness of such weak solutions, see Assumption~\ref{AssumpUniqueWeakSolu}
and \cite[Section 8 ]{jara2008hydrodynamic}, implies the uniqueness of the limit point $\pi_t$.

\subsection{Proof of fluctuations}

\subsubsection{Consequences of self-duality: fluctuations}
\label{ConsDualFluc}
We first extend the results of Proposition~\ref{Growthcontrolone} and Proposition~\ref{Growthcontroltwo} to
correlations of order three and four.

\begin{proposition}
  \label{GrowthcontrolFluc}
  Under Assumption~\ref{HydroAssump} on the initial distribution of particles, we have that there exists a
  constant $C>0$ depending on $\rho_0$ and $C_i$ for $i \in \lbrace 1,2,3 \rbrace$, but not in $n$, such that:
\begin{align}
&\E_n \left[   \eta_{r(n)t}(x/n) \, \eta_{r(n)t}(w/n) \, \eta_{r(n)t}(z/n)  \right] \nn \\
&\leq   \,\sum_{x\myprime, w\myprime, z\myprime} p_{r(n)t}^{(3)}(x/n,w/n, z/n; x\myprime/n, w\myprime/n, z\myprime/n)\rho_0(x\myprime/n) \rho_0(w\myprime/n) \rho_0(z\myprime/n) \nn \\
 &+ \frac{1}{\alpha}  \1_{x=w}  \sum_{x\myprime, w\myprime, z\myprime} p_{r(n)t}^{(3)}(x/n,x/n, z/n; x\myprime/n, w\myprime/n, z\myprime/n)\rho_0(x\myprime/n) \rho_0(w\myprime/n) \rho_0(z\myprime/n) \nn \\
&+  \1_{x=w}  \, \sum_{x\myprime, z\myprime} p_{r(n)t}^{(2)}(x/n,z/n; x\myprime/n,z\myprime/n)\rho_0(x\myprime/n)  \rho_0(z\myprime/n) \nn \\
 &+ \frac{1}{\alpha}  \1_{x=z}  \sum_{x\myprime, w\myprime, z\myprime} p_{r(n)t}^{(3)}(x/n,w/n, x/n; x\myprime/n, w\myprime/n, z\myprime/n)\rho_0(x\myprime/n) \rho_0(w\myprime/n) \rho_0(z\myprime/n)\nn \\
&+\1_{x=z}  \, \sum_{x\myprime, w\myprime} p_{r(n)t}^{(2)}(x,w; x\myprime,w\myprime)\rho_0(x\myprime/n) \rho_0(w\myprime/n) + \frac{C}{n},
\end{align}
for all $t \geq 0$, and all $x,w, z \in \Zd$ such that $w \neq z$.
\end{proposition}

\begin{proof}
For the three-points correlations, first notice that by self-duality
\begin{align}
 &\E_n \left[  \eta_{r(n)t}(x/n)  \eta_{r(n)t}(w/n)\eta_{r(n)t}(z/n)  \right] \nn \\
 &= \alpha^3 \1_{x\neq w} \1_{x\neq z} \,\sum_{x\myprime, w\myprime, z\myprime} p_{r(n)t}^{(3)}(x/n,w/n, z/n; x\myprime/n, w\myprime/n, z\myprime/n)\int_\Omega D(\delta_{x\myprime/n}+\delta_{w\myprime/n}+\delta_{z\myprime/n},\eta) \, \nu_{\rho_0}^n(d\eta) \nn \\
 &+ \alpha^2 (\alpha+1) \1_{x=w}  \sum_{x\myprime, w\myprime, z\myprime} p_{r(n)t}^{(3)}(x/n,x/n, z/n; x\myprime/n, w\myprime/n, z\myprime/n)\int_\Omega D(\delta_{x\myprime/n}+\delta_{w\myprime/n}+\delta_{z\myprime/n},\eta) \, \nu_{\rho_0}^n(d\eta) \nn \\
&+ \alpha^2 \1_{x=w}  \, \sum_{x\myprime, z\myprime} p_{r(n)t}^{(2)}(x/n,z/n; x\myprime/n,z\myprime/n)\int_\Omega D(\delta_{x\myprime/n}+\delta_{z\myprime/n},\eta) \, \nu_{\rho_0}^n(d\eta) \nn \\
 &+ \alpha^2 (\alpha+1) \1_{x=z}  \sum_{x\myprime, w\myprime, z\myprime} p_{r(n)t}^{(3)}(x/n,w/n, x/n; x\myprime/n, w\myprime/n, z\myprime/n)\int_\Omega D(\delta_{x\myprime/n}+\delta_{w\myprime/n}+\delta_{z\myprime/n},\eta) \, \nu_{\rho_0}^n(d\eta)\nn \\
&+ \alpha^2 \1_{x=z}  \, \sum_{x\myprime, w\myprime} p_{r(n)t}^{(2)}(x,w; x\myprime,w\myprime)\int_\Omega D(\delta_{x\myprime/n}+\delta_{w\myprime/n},\eta) \, \nu_{\rho_0}^n(d\eta).
\end{align}
We can estimate the two-point sums with~\eqref{DsecondorderEstimate}. For the three-point sums, we proceed
  analogously to~\eqref{DsecondorderEstimate} and apply Assumption~\ref{HydroAssump} multiple times as
  follows: \small 
\begin{align}
  \label{DthirdorderEstimate}
&\alpha^3\int_\Omega D(\delta_{x\myprime/n}+\delta_{w\myprime/n}+\delta_{z\myprime/n},\eta) \, \nu_{\rho_0}^n(d\eta) \nn \\
&\leq \alpha^3\left| \int_\Omega D(\delta_{x\myprime/n}+\delta_{w\myprime/n}+\delta_{z\myprime/n},\eta) \, \nu_{\rho_0}^n(d\eta) -\int_\Omega D(\delta_{x\myprime/n},\eta) \, \nu_{\rho_0}^n(d\eta) \int_\Omega D(\delta_{w\myprime/n},\eta) \, \nu_{\rho_0}^n(d\eta) \int_\Omega D(\delta_{z\myprime/n},\eta) \, \nu_{\rho_0}^n(d\eta)  \right| \nn \\
&+ \left|\alpha^2 \int_\Omega D(\delta_{w\myprime/n},\eta) \, \nu_{\rho_0}^n(d\eta) \int_\Omega D(\delta_{z\myprime/n},\eta) \, \nu_{\rho_0}^n(d\eta)   \right| \left|\alpha \int_\Omega D(\delta_{x\myprime/n},\eta) \, \nu_{\rho_0}^n(d\eta) - \rho_0(x\myprime/n) \right| \nn \\
&+ \left| \alpha \rho_0(x\myprime/n) \, \int_\Omega D(\delta_{z\myprime/n},\eta) \, \nu_{\rho_0}^n(d\eta)  \right| \left| \alpha  \int_\Omega D(\delta_{w\myprime/n},\eta) \, \nu_{\rho_0}^n(d\eta)  -  \rho_0(w\myprime/n) \right| \nn \\
&+ \left|  \rho_0(x\myprime/n) \,  \rho_0(w\myprime/n)  \right| \left| \alpha  \int_\Omega D(\delta_{z\myprime/n},\eta) \, \nu_{\rho_0}^n(d\eta)  -  \rho_0(z\myprime/n) \right| \nn \\
&+ \rho_0(x\myprime/n) \rho_0(w\myprime/n) \rho_0(z\myprime/n) \nn \\
&\leq \frac{C_3}{n} + \left|\alpha^2 \int_\Omega D(\delta_{w\myprime/n},\eta) \, \nu_{\rho_0}^n(d\eta) \int_\Omega D(\delta_{z\myprime/n},\eta) \, \nu_{\rho_0}^n(d\eta)   \right| \frac{C_1}{n} +  \rho_0(x\myprime/n) \, \left|  \alpha \int_\Omega D(\delta_{z\myprime/n},\eta) \, \nu_{\rho_0}^n(d\eta)  \right| \frac{C_1}{n} \nn \\
&+  \rho_0(x\myprime/n) \,  \rho_0(w\myprime/n)  \frac{C_1}{n} + \rho_0(x\myprime/n) \rho_0(w\myprime/n) \rho_0(z\myprime/n). 
\end{align} 
\normalsize 
To simplify further,  we apply Assumption~\ref{HydroAssump} again,
\begin{align}
  \label{DthirdorderEstimatetwo}
&\alpha^3\int_\Omega D(\delta_{x\myprime/n}+\delta_{w\myprime/n}+\delta_{z\myprime/n},\eta) \, \nu_{\rho_0}^n(d\eta) \nn \\
&\leq \frac{C_3}{n} + \left|\alpha^2 \int_\Omega D(\delta_{w\myprime/n},\eta) \, \nu_{\rho_0}^n(d\eta) \int_\Omega D(\delta_{z\myprime/n},\eta) \, \nu_{\rho_0}^n(d\eta)  -\alpha \rho_0(w\myprime/n) \int_\Omega D(\delta_{z\myprime/n},\eta) \, \nu_{\rho_0}^n(d\eta)  \right| \frac{C_1}{n} \nn \\
&+ \left|\alpha \rho_0(w\myprime/n) \int_\Omega D(\delta_{z\myprime/n},\eta) \, \nu_{\rho_0}^n(d\eta) -\rho_0(w\myprime/n) \rho_0(z\myprime/n) \right| \frac{C_1}{n} +  \rho_0(w\myprime/n) \,  \rho_0(z\myprime/n)  \frac{C_1}{n} \nn \\
&+  \rho_0(x\myprime/n) \, \left|  \alpha \int_\Omega D(\delta_{z\myprime/n},\eta) \, \nu_{\rho_0}^n(d\eta)  -\rho_0(z\myprime/n)\right| \frac{C_1}{n} +\rho_0(x\myprime/n) \, \rho_0(z\myprime/n)  \frac{C_1}{n}\nn \\
&+  \rho_0(x\myprime/n) \,  \rho_0(w\myprime/n)  \frac{C_1}{n} + \rho_0(x\myprime/n) \rho_0(w\myprime/n) \rho_0(z\myprime/n) \nn \\
&\leq \frac{C_3}{n} + \frac{C_1^3}{n^{3}}+\rho_0(z\myprime/n) \frac{C_1^2}{n^{2}} + \rho_0(w\myprime/n) \frac{C_1^2}{n^{2}} +  \rho_0(w\myprime/n) \,  \rho_0(z\myprime/n)  \frac{C_1}{n} \nn \\
&+  \rho_0(x\myprime/n) \frac{C_1^2}{n^{2}}+\rho_0(x\myprime/n) \rho_0(z\myprime/n)  \frac{C_1}{n}+  \rho_0(x\myprime/n) \,  \rho_0(w\myprime/n)  \frac{C_1}{n} + \rho_0(x\myprime/n) \rho_0(w\myprime/n) \rho_0(z\myprime/n). 
\end{align} 
\begin{remark}
  The previous expression has structure we do not use; for us it suffices to have the bound
  \begin{equation}
    \rho_0(x\myprime/n) \rho_0(w\myprime/n) \rho_0(z\myprime/n) + \frac{C(\norm{\rho}_\infty,C_1,C_2,C_3)}{n},
  \end{equation} 
  where $C(\norm{\rho}_\infty,C_1,C_2,C_3)$ is a constant that does not depend on $n$.
\end{remark} 
\end{proof}

In the same way we can also bound four points correlations.

\begin{proposition}
  \label{GrowthcontrolFluc4points}
  Under Assumption~\ref{HydroAssump} on the initial distribution of particles, there exists a constant $C>0$
  which depends on $\rho_0$ and $C_i$ for $i \in \lbrace 1,2,3,4 \rbrace$, but does not depend on $n$,
  such that \small
\begin{align}
\E_n &\left[  \eta_{r(n)t}(x/n)\eta_{r(n)t}(y/n)  \eta_{r(n)t}(w/n)\eta_{r(n)t}(z/n) \right] \nn \\
 &\leq \sum_{x\myprime, y\myprime,w\myprime, z\myprime} p_{r(n)t}^{(4)}(x/n,y/n,w/n,z/n; x\myprime/n, y\myprime/n,w\myprime/n, z\myprime/n)  \rho_0(x\myprime/n) \rho_0(y\myprime/n) \rho_0(w\myprime/n) \rho_0(z\myprime/n) \nn \\
 &+ \1_{x=w} \1_{y\neq z} \frac{2}{\alpha}   \sum_{x\myprime, y\myprime,w\myprime, z\myprime}  p_{r(n)t}^{(4)}(x/n,y/n,x/n,z/n; x\myprime/n, y\myprime/n,w\myprime/n, z\myprime/n)  \rho_0(x\myprime/n) \rho_0(y\myprime/n) \rho_0(w\myprime/n) \rho_0(z\myprime/n) \nn \\ 
&+ \1_{x=w} \1_{y\neq z} \sum_{x\myprime, y\myprime,w\myprime, z\myprime}  p_{r(n)t}^{(3)}(x,y,z; x\myprime, y\myprime, z\myprime)  \rho_0(x\myprime/n) \rho_0(y\myprime/n) \rho_0(z\myprime/n) \nn \\ 
&+ \1_{x=w} \1_{y=z} \frac{2(2\alpha+1)}{\alpha^2}  \sum_{x\myprime, y\myprime,w\myprime, z\myprime} p_{r(n)t}^{(4)}(x/n,y/n,x/n,y/n; x\myprime/n, y\myprime/n,w\myprime/n, z\myprime/n)  \rho_0(x\myprime/n) \rho_0(y\myprime/n) \rho_0(w\myprime/n) \rho_0(z\myprime/n) \nn \\ 
&+ \1_{x=w} \1_{y=z} \frac{2(\alpha+1)}{\alpha}  \sum_{x\myprime, y\myprime,w\myprime, z\myprime}  p_{r(n)t}^{(3)}(x/n,y/n,x/n; x\myprime/n, y\myprime/n,w\myprime/n)  \rho_0(x\myprime/n) \rho_0(y\myprime/n) \rho_0(w\myprime/n) \nn \\ 
&+ \1_{x=w} \1_{y=z} \frac{2(\alpha+1)}{\alpha}  \sum_{x\myprime, y\myprime,w\myprime, z\myprime}  p_{r(n)t}^{(3)}(x/n,y/n,y/n; x\myprime/n, y\myprime/n,w\myprime/n)  \rho_0(x\myprime/n) \rho_0(y\myprime/n) \rho_0(w\myprime/n)  \nn \\ 
&+ \1_{x=w} \1_{y=z} 2   \sum_{x\myprime, y\myprime,w\myprime, z\myprime}  p_{r(n)t}^{(2)}(x/n,y/n; x\myprime/n, y\myprime/n)  \rho_0(x\myprime/n) \rho_0(y\myprime/n)  + \frac{C}{n},
\end{align}
\normalsize
for all $t \geq 0$, and all $x,y,w, z \in \Zd$ such that $x\neq y$ and $w \neq z$.
\end{proposition}

\begin{proof} 
The proof is lengthy but analogous to the one of Proposition~\ref{GrowthcontrolFluc}. 
\end{proof}

\subsection{Dynkin Martingales: fluctuations}

In this part of the proof of Theorem~\ref{FluctTheorem}, we 
obtain a result analogous to Lemma~\ref{LemmaDynkin}, with the difference that this time the quadratic
variation of the associated Dynkin martingales will not vanish as $n$ goes to infinity.

\subsubsection*{Closing the drift}

First we deal with the drift term.

\begin{lemma}
  \label{LemmaDrift}
  For all $\phi \in S(\Rd)$ and for all $t \in [0,T]$,
  \begin{equation}
    N_t^{(n)}(\phi) = Y_t^{(n)}(\phi,\eta) - Y_0^{(n)}(\phi,\eta) -\int_0^t Y_s^{(n)}(L_n^{(1)} \phi,\eta) \, ds
  \end{equation} 
  is an $\caF_t$-martingale. 
\end{lemma} 

\begin{proof}
By Dynkin's formula we have that
\begin{equation}
  N_t^{(n)}(\phi) = Y_t^{(n)}(\phi,\eta) - Y_0^{(n)}(\phi,\eta) -\int_0^t
  \left( \loc_n +\partial_s \right) Y_s^{(n)}( \phi,\eta) \, ds
\end{equation} 
is an $\caF_t$-martingale.

By self-duality with one particle we have
\begin{align}
  \label{closingdriftfirststep}
  \loc_n Y_t^{(n)}(\phi,\eta) &= n^{d/2} \loc_n \pi_t^{(n)}(\phi,\eta) \nn \\
                              &= n^{d/2} \pi_t^{(n)}(L_n^{(1)} \phi, \eta),
\end{align}
where in the first equality we used that the expectation in the definition of $Y_t^{(n)}$ does not depend on $\eta$, and in the second equality we used~\eqref{ActionGenOnepart}.\\

In order close the drift, i.e. write $\loc_n Y_t^{(n)}(\phi,\eta)$  as $Y_t^{(n)}(L_n^{(1)} \phi,\eta)$, we need to substract
\begin{equation}
  \E_n \left[ n^{d/2} \pi_t^{(n)}(L_n^{(1)} \phi, \eta)  \right]. \nn 
\end{equation} 
However, again by duality, we have
\begin{align}
 \E_n \left[ n^{d/2} \pi_t^{(n)}(L_n^{(1)} \phi, \eta)  \right] &= \frac{\alpha}{n^{d/2}}  \sum_{x \in \Zd} \int_\Omega \E_{x/n} \left[ D(\delta_{X_t/n},\eta)  \right] \nu_{\rho_0}(d\eta) L_n^{(1)} \phi(x/n)  \nn \\
&= \frac{\alpha}{n^{d/2}}  \sum_{ x, x\myprime \in \Zd} p_{r(n)t}^{(1)}(x/n,x\myprime/n) \nu_{\rho_0}(x\myprime/n) L_n^{(1)} \phi(x/n) \nn \\
&= \frac{1}{n^{d/2}}  \sum_{ x \in \Zd} L_n^{(1)}\left( \alpha \sum_{ x\myprime \in \Zd} p_{r(n)t}^{(1)}(x/n,x\myprime/n) \nu_{\rho_0}(x\myprime/n)\right) \phi(x/n) \nn \\
&= \frac{1}{n^{d/2}}  \sum_{ x \in \Zd} \partial_t \left( \alpha \sum_{ x\myprime \in \Zd} p_{r(n)t}^{(1)}(x/n,x\myprime/n) \nu_{\rho_0}(x\myprime/n)\right) \phi(x/n) \nn \\
&= - \partial_t Y_t^{(n)}(\phi,\eta),
\end{align}
where $\nu_{\rho_0}(x\myprime/n) := \int \eta(x\myprime/n) \nu_{\rho_0}(d \eta)$, in the third line we used
reversibility of $L_n^{(1)}$ with respect to the counting measure in $\Zd_n$, then in the fourth-line the
Hille-Yoshida Theorem, and in the last line we simply identified the action of the operator $\partial_t$ on
the fluctuation field.
\end{proof} 

\subsubsection*{Computing the quadratic variation}

The following proposition identifies the limiting mean of the quadratic variation of the martingales
$N_t^{(n)}$ of Lemma~\ref{LemmaDrift}.
\begin{lemma}
  \label{ConvQuadraVarFluc}
  For all $\phi \in S(\Rd)$, and for all $t \in [0,T]$
  \begin{align}
    \label{ExpecQuadLimit}
    \lim_{n \to \infty} \E_n \left[  \int_0^t \Gamma_n (Y_s^{(n)}(\phi,\eta)) \, ds   \right]
    &= \int_0^t \norm{ \sqrt{\Gamma_s^{\rho} \phi}}^2_{L^2(\Rd, \rho(s,x) \, dx)} \, ds,
  \end{align}
  where $\rho(t,x)$ solves~\eqref{HydroEquationLJ}. Moreover, we also have
  \begin{align}
    \label{VanishSecondQuad}
    \lim_{n \to \infty} \E_n \left[  \int_0^t \left(\Gamma_n (Y_s^{(n)}(\phi,\eta))
    -\norm{ \sqrt{\Gamma_s^{\rho}\phi}}^2_{L^2(\Rd, \rho(s,x) \, dx)} \right)^2 \, ds   \right] &= 0.
  \end{align}
\end{lemma}

\begin{proof} 
  First notice that by the definition of the Carré du champ operator, and the same computations as in the
  proofs of Propositions~\ref{Growthcontrolone} and~\ref{Growthcontroltwo}, we have \small
\begin{align}
&\E_n \left[   \Gamma_n (Y_t^{(n)}(\phi,\eta))      \right] \nonumber \\
&= \frac{\alpha^2}{n^{2d}}\sum_{x, y \in \Zd} q((y-x)/n) \,  \sum_{x\myprime \in \Zd} p_{r(n)t}(x/n, x\myprime/n) \int_\Omega D(\delta_{x\myprime/n},\eta) \nu_{\rho}(d\eta) \left( \phi(y/n) - \phi(x/n) \right)^2 \nn \\
&+\frac{\alpha^2}{n^{2d}} \sum_{x, y \in \Zd} q((y-x)/n) \,  \sum_{x\myprime, y\myprime \in \Zd} p_{r(n)t}^{(2)}(x/n,y/n ; x\myprime/n, y\myprime/n)\int_\Omega D(\delta_{x\myprime/n} +\delta_{y\myprime},\eta) \nu_{\rho}(d\eta)  \left( \phi(y/n) - \phi(x/n) \right)^2. \nn 
\end{align}
\normalsize
By Mosco convergence of one and two $\SIP$-particles, we also have that the quantity
\small
\begin{align}
\caM_t^n(\phi) &:= \frac{\alpha}{n^{2d}}\sum_{x, y \in \Zd} q((y-x)/n) \,  \sum_{x\myprime \in \Zd} p_{r(n)t}(x/n, x\myprime/n) \rho_0(x\myprime/n) \left( \phi(y/n) - \phi(x/n) \right)^2 \nn \\
&+\frac{1}{n^{2d}} \sum_{x, y \in \Zd} q((y-x)/n) \,  \sum_{x\myprime, y\myprime \in \Zd} p_{r(n)t}^{(2)}(x/n,y/n ; x\myprime/n, y\myprime/n) \rho_0(x\myprime/n) \rho_0(y\myprime/n)  \left( \phi(y/n) - \phi(x/n) \right)^2 \nn 
\end{align}
\normalsize
converges to 
\begin{equation}
  \norm{ \sqrt{\Gamma_s^{\rho}\phi}}^2_{L^2(\Rd, \rho(t,x) \, dx)}.
\end{equation} 
In order to show~\eqref{ExpecQuadLimit}, it is then enough to show
\begin{align}
  \label{VanishmeanAbsDiff}
\lim_{n \to \infty}   \left| \E_n \left[ \Gamma_n (Y_t^{(n)}(\phi,\eta)) \right]-\caM_t^n(\phi) \right|   = 0.
\end{align}
However, this is a consequence of Assumption~\ref{HydroAssump}, and Mosco convergence of one and two particles.

Now, to prove~\eqref{VanishSecondQuad} we show 
\begin{align}
  \lim_{n \to \infty} \E_n \left[  \left( \int_0^t \Gamma_n (Y_s^{(n)}(\phi,\eta)) -\caM_s^n(\phi) \, ds  \right)^2 \right] &= 0.
\end{align}
We then use the Cauchy-Schwarz inequality to obtain
\begin{align}
  \E_n &\left[ \left(  \int_0^t \Gamma_n (Y_s^{(n)}(\phi,\eta)) -\caM_s^n(\phi)\, ds \right)^2   \right]  &\leq t \, \E_n \left[ \int_0^t \left(   \Gamma_n (Y_s^{(n)}(\phi,\eta)) -\caM_s^n(\phi)    \right)^2 \, ds    \right].
\end{align}
It is then enough to show that the expectation of the integrand of the right-hand side
vanishes as $n$ goes to infinity. That is, we want to show that the following term vanishes: 
\begin{align}
  \label{VanishingCarre}
  &\E_n \left[ \left(   \Gamma_n (Y_t^{(n)}(\phi,\eta)) -\caM_t^n(\phi) \right)^2 \right] \nn \\  
  &= \E_n \left[ \left(   \Gamma_n (Y_t^{(n)}(\phi,\eta))  \right)^2 \right] + \left( \caM_t^n(\phi) \right)^2  -2 \caM_t^n(\phi) \E_n \left[  \Gamma_n (Y_t^{(n)}(\phi,\eta))   \right].
\end{align}
We split the term with the minus sign in halfs and first deal with
\begin{align}
  &\caM_t^n(\phi) \left( \caM_t^n(\phi)  -  \E_n \left[  \Gamma_n (Y_t^{(n)}(\phi,\eta))   \right]  \right)\nn 
  &\leq \caM_t^n(\phi) \left| \caM_t^n(\phi)  -  \E_n \left[  \Gamma_n (Y_t^{(n)}(\phi,\eta))   \right]  \right|,
\end{align}
which by \eqref{VanishmeanAbsDiff} vanishes as $n$ goes to infinity.

We are left to show that
\begin{equation}
  \label{ECarreSquareMinusMtimesExp} 
  \lim_{n \to \infty} \E_n \left[ \left(   \Gamma_n (Y_t^{(n)}(\phi,\eta))  \right)^2 \right]
  - \caM_t^n(\phi) \E_n \left[  \Gamma_n (Y_t^{(n)}(\phi,\eta))   \right] = 0.
\end{equation} 
Notice that
\small
\begin{align}
  \label{ExpCarreSquare}
  &\E_n \left[ \left(   \Gamma_n (Y_t^{(n)}(\phi,\eta))  \right)^2 \right] \nn \\
  &= \frac{\alpha^2}{n^{2d}} \sum_{x,w}  \E_{\nu_\rho^n} \left[ \eta_{r(n)t}(x/n) \eta_{r(n)t}(w/n) \right] \Gamma_n^{(1)}\phi(x/n) \,\Gamma_n^{(1)} \phi(w/n) \nn \\ 
  &+ \frac{1}{n^{4d}} \sum_{x,y,z,w} q(x-y) q(w-z)  \E_{\nu_\rho^n} \left[ \eta_{r(n)t}(x/n) \eta_{r(n)t}(y/n) \eta_{r(n)t}(w/n) \eta_{r(n)t}(z/n) \right] \left( \nabla_n^{(x,y)}\phi \right)^2  \, \left( \nabla_n^{(w,z)}\phi \right)^2 \nn \\ 
  &+ \frac{2 \alpha}{n^{3d}} \sum_{x,z,w} q(w-z)   \E_{\nu_\rho^n} \left[ \eta_{r(n)t}(x/n)  \eta_{r(n)t}(w/n) \eta_{r(n)t}(z/n) \right] \Gamma_n^{(1)}\phi(x/n) \,  \left( \nabla_n^{(w,z)}\phi \right)^2, 
\end{align} 
\normalsize
and 
\small 
\begin{align}
  \label{ExpCarretimesMn}
& \caM_t^n(\phi) \E_n \left[  \Gamma_n (Y_t^{(n)}(\phi,\eta))   \right] \nn \\
&= \frac{\alpha^2}{n^{2d}} \sum_{x,w} \,  \sum_{x\myprime \in \Zd} p_{r(n)t}(x/n, x\myprime/n) \rho_0(x\myprime/n) \E_{\nu_\rho^n} \left[ \eta_{r(n)t}(w/n) \right] \Gamma_n^{(1)}\phi(x/n) \,\Gamma_n^{(1)} \phi(w/n) \nn \\ 
&+ \frac{1}{n^{4d}} \sum_{x,y,z,w} q(x-y) q(w-z)  \sum_{x\myprime, y\myprime \in \Zd} p_{r(n)t}^{(2)}(x/n,y/n ; x\myprime/n, y\myprime/n) \rho_0(x\myprime/n) \rho_0(y\myprime/n) \nn \\
&\times \E_{\nu_\rho^n} \left[  \eta_{r(n)t}(w/n) \eta_{r(n)t}(z/n) \right] \left( \nabla_n^{(x,y)}\phi \right)^2  \, \left( \nabla_n^{(w,z)}\phi \right)^2 \nn \\ 
&+ \frac{ \alpha}{n^{3d}} \sum_{x,z,w} q(w-z)    \sum_{x\myprime \in \Zd} p_{r(n)t}(x/n, x\myprime/n) \rho_0(x\myprime/n) \E_{\nu_\rho^n} \left[  \eta_{r(n)t}(w/n) \eta_{r(n)t}(z/n) \right] \Gamma_n^{(1)}\phi(x/n) \,  \left( \nabla_n^{(w,z)}\phi \right)^2 \nn \\ 
&+ \frac{ \alpha}{n^{3d}} \sum_{x,z,w} q(w-z)   \E_{\nu_\rho^n} \left[ \eta_{r(n)t}(x/n)   \right] \sum_{w\myprime, z\myprime \in \Zd} p_{r(n)t}^{(2)}(w,z ; w\myprime, z\myprime) \rho_0(w\myprime/n) \rho_0(z\myprime/n)\Gamma_n^{(1)}\phi(x/n) \,  \left( \nabla_n^{(w,z)}\phi \right)^2, 
\end{align}
\normalsize where $\nabla_n^{(w,z)}\phi:= \phi(x/n)-\phi(z/n)$. To show~\eqref{ECarreSquareMinusMtimesExp}, it
is enough to compare one by one the quantities in~\eqref{ExpCarreSquare} minus the ones
in~\eqref{ExpCarretimesMn}. For example, we estimate first the difference of the first term in each of the
right hand side of~\eqref{ExpCarreSquare} and~\eqref{ExpCarretimesMn}, namely
\begin{align*}
\Xi_n^{(1)}(\eta,t) &:= \frac{\alpha^2}{n^{2d}} \sum_{x,w}  \E_{\nu_\rho^n} \left[ \eta_{r(n)t}(x/n) \eta_{r(n)t}(w/n) \right] \Gamma_n^{(1)}\phi(x/n) \,\Gamma_n^{(1)} \phi(w/n) \nn \\
&- \frac{\alpha^2}{n^{2d}} \sum_{x,w} \,  \sum_{x\myprime \in \Zd} p_{r(n)t}(x/n, x\myprime/n) \rho_0(x\myprime/n) \E_{\nu_\rho^n} \left[ \eta_{r(n)t}(w/n) \right] \Gamma_n^{(1)}\phi(x/n) \,\Gamma_n^{(1)} \phi(w/n) \nn \\
&= \frac{\alpha^2}{n^{2d}} \sum_{x,w} S_n^{(2)}(t) \left( \rho_0(x/n) \rho_0(w/n)\right) \Gamma_n^{(1)}\phi(x/n) \,\Gamma_n^{(1)} \phi(w/n) \nn \\
&-\frac{\alpha^2}{n^{2d}} \sum_{x,w} S_n^{(1)}(t) \left( \rho_0(x/n) \right)  S_n^{(1)}(t) \left(\rho_0(w/n)\right) \Gamma_n^{(1)}\phi(x/n) \,\Gamma_n^{(1)} \phi(w/n) + O(\tfrac1 n).
\end{align*}
where in the last equality we used Assumption \ref{HydroAssump} twice. Notice that by
Proposition~\ref{StrongConvLn1Gamman1}, $\Gamma_n^{(1)}\phi$ converges strongly to
$\Gamma^{(1)}\phi$. Moreover, the semigroup $S_n^{(2)}(t)$ converges strongly to the semigroup corresponding
to two independent copies of the semigroup associated with the generator $L^{(1)}$. Hence 
\begin{equation*}
    \lim_{n \to \infty } \Xi_n^{(1)}(\eta,t) = 0.
\end{equation*}
The remaining pairs of  
terms in the right-hand sides of~\eqref{ExpCarreSquare} and~\eqref{ExpCarretimesMn} are analyzed analogously,
using Proposition~\ref{StrongConvLn1Gamman1}, Assumption~\ref{HydroAssump}, and the factorization in the limit
of the semigroups $S_n^{(k)}(t)$ for $k \in \lbrace 2, 3, 4\rbrace$.
\end{proof}

\subsubsection{Tightness: fluctuations}
We show tightness using Mitoma's criterion. 
This criterion allows us to simplify the situation to the real-valued scenario. Moreover, thanks to
Proposition~\ref{LemmaDrift}, it is enough to show tightness for each of the three processes
\begin{equation*}
  \left\{\, Y_0^{(n)}(\phi,\eta) \;\middle|\; n \in \N \, \right\}, \quad
  \left\{\, N_t^{(n)}(\phi,\eta) \;\middle|\; t \in [0,T], n \in \N \, \right\}, 
\end{equation*}
and
\begin{equation*}
  \left\{\, \int_0^t Y_s^{(n)}(L_n^{(1)} \phi,\eta) \, ds \;\middle|\; t \in [0,T], n \in \N \, \right\} .
\end{equation*}
We will work separately on each process.

\subsubsection*{Tightness for the fluctuation field at time zero}

We consider the first process. Here it is just enough to notice that by the same computations as those
in~\eqref{twopointscorr},
\begin{align}
\E_{\nu_\rho^n} \left[ Y_0^{(n)}(\phi, \eta)^2 \right] &= \frac{1}{n^d} \sum_{x,y \in \Zd} \left( \E_{\nu_\rho^n} \left[ \eta(x/n) \eta(y/n) \right] -\rho_0^n(x/n) \rho_0^n(y/n) \right)  \phi(x/n) \, \phi(y/n) \nn \\
&= \frac{1}{n^{d}} \sum_{x \in \Zd} \left( \int_\Omega D(2 \delta_{x/n} ,\eta) \, \nu_{\rho_0}^n(d\eta) -\rho_0^n(x)^2 \right)  \phi(x/n)^2  \nn \\
&{}\quad+\frac{\alpha (\alpha+1)}{n^d} \sum_{x \in \Zd}  \int_\Omega D(2\delta_{x/n},\eta) \, \nu_{\rho_0}^n(d\eta)  \phi(x/n)^2  \nn \\
&{}\quad+ \frac{\alpha}{n^d} \sum_{x \in \Zd}  \int_\Omega D(\delta_{x/n},\eta) \, \nu_{\rho_0}^n(d\eta) \phi(x/n)^2,
\end{align}
which by Assumption~\ref{HydroAssump} is bounded.

\subsubsection*{Tightness of Martingales}
For the second process, we follow the proof of~\cite[Lemma 5.2]{erhard2020non}, which shows the
convergence of a similar martingale to a mean zero Gaussian process of quadratic variation of time-integral
form. From the same arguments in~\cite{erhard2020non} it is enough to check that the following holds.
\begin{enumerate}
\item Jumps size:
  \begin{equation}
    \label{JumpsControlQVNtn} 
    \lim_{n \to \infty} \E_n \left[ \sup_{s \leq t} \left|    N_s^{(n)}(\phi,\eta)- N_{-s}^{(n)}(\phi,\eta) \right| \right] = 0.
  \end{equation} 
\item Convergence in probability of the quadratic variation process:
  \begin{equation}
    \label{ConvQVNtProb} 
    \lim_{n \to \infty} \P_n \left( \left| \int_0^t \Gamma_n (Y_s^{(n)}(\phi,\eta)) \, ds -  \int_0^t \norm{ \sqrt{\Gamma_s^{\rho}f}}^2_{L^2(\Rd, \rho(s,x) \, dx)} \, ds  \right| > \epsilon  \right) = 0,
  \end{equation} 
  where $\rho(t,x)$ solves \eqref{HydroEquationLJ}.
\end{enumerate}
To see that~\eqref{JumpsControlQVNtn} holds, it is enough to notice that the jumps of the process
$\{ \eta(t): t \geq 0 \}$ are determined by exponential clocks. This implies that, for any $\delta>0$, the
probability of having more than one jump in the interval $(s,s+\delta]$ is of the order $o(\delta)$. Hence for
$C$, a positive constant that depends on the model parameters, we have \be \sup_{s \leq t} \left|
  N_s^{(n)}(\phi,\eta)- N_{-s}^{(n)}(\phi,\eta) \right| = \sup_{s \leq t} \left| Y_s^{(n)}(\phi,\eta)-
  Y_{-s}^{(n)}(\phi,\eta) \right| \leq C \frac{\norm{\phi}_{\infty}}{n^{d/2}},
\end{equation}
with probability $1-o(\delta)$.

For the convergence of the quadratic variation process, we obtain convergence in distribution by
Proposition~\ref{ConvQuadraVarFluc}. The fact that the limit is deterministic upgrades this to convergence in
probability.

\subsubsection*{Tightness of integral term}
For the third term, the drift term, we use a combination of Theorem 1.3 in \cite[Chapter IV]{kipnis1998scaling} and Aldous’ criterion (in the form given in Proposition 1.6 of \cite[Chapter IV]{kipnis1998scaling}). In order to verify part 1) of Theorem 1.3 of \cite{kipnis1998scaling}, it is enough to estimate 
\begin{align}
    \E_n &\left[ \sup_{t \in [0,T] }  \left( \int_0^t Y_s^{(n)}(L_n^{(1)} \phi,\eta) \, ds \)^ 2  \right] \leq T \int_0^T \E_n \left[  \left(  Y_s^{(n)}(L_n^{(1)} \phi,\eta) \)^2  \right] \, ds \nn \\
    &= T \int_0^T  \frac{1}{n^d} \sum_{x,y \in \Zd} \left( \E_{\nu_\rho^n} \left[ \eta_{r(n)t}(x/n) \eta_{r(n)t}(y/n) \right] -\rho_t^n(x/n) \rho_t^n(y/n) \right) L_n^{(1)}\phi(x/n) \,L_n^{(1)} \phi(y/n) \, dt.
\end{align}
By the exact computations as in \eqref{twopointscorr} we obtain
\begin{align}
    \E_n &\left[ \sup_{t \in [0,T] }  \left( \int_0^t Y_s^{(n)}(L_n^{(1)} \phi,\eta) \, ds \)^ 2  \right] \leq T \int_0^T \E_n \left[  \left(  Y_s^{(n)}(L_n^{(1)} \phi,\eta) \)^2  \right] \, ds \nn \\
    &= T \int_0^T  \frac{\alpha (\alpha+1)}{n^d} \sum_{x \in \Zd} \sum_{x\myprime, y\myprime \in \Zd} p_{r(n)t}^{(2)}(x/n,x/n; x\myprime/n, y\myprime/n) \int_\Omega D(\delta_{x\myprime/n}+\delta_{y\myprime/n},\eta) \, \nu_{\rho_0}^n(d\eta)  L_n^{(1)}\phi(x/n)^2 \, dt \nn \\
    &+ T \int_0^T  \frac{\alpha}{n^d} \sum_{x \in \Zd} \sum_{x\myprime \in \Zd} p_{r(n)t}^{(1)}(x/n,x\myprime/n) \int_\Omega D(\delta_{x\myprime/n},\eta) \, \nu_{\rho_0}^n(d\eta)  L_n^{(1)}\phi(x/n)^2  \, dt \nn \\
    &\leq T \int_0^T  \frac{(\alpha+1)}{\alpha n^d} \sum_{x \in \Zd} \sum_{x\myprime, y\myprime \in \Zd} p_{r(n)t}^{(2)}(x/n,x/n; x\myprime/n, y\myprime/n) \rho_0(x\myprime/n) \rho_0(y\myprime/n)  L_n^{(1)}\phi(x/n)^2 \, dt \nn \\
    &+ T \int_0^T  \frac{1}{n^d} \sum_{x \in \Zd} \sum_{x\myprime \in \Zd} p_{r(n)t}^{(1)}(x/n,x\myprime/n) \rho_0(x\myprime/n) L_n^{(1)}\phi(x/n)^2  \, dt \nn \\
    &+ T \int_0^T  \frac{ C_1 (\alpha+1)}{\alpha n^{2d}} \sum_{x \in \Zd}   L_n^{(1)}\phi(x/n)^2 \, dt + T \int_0^T  \frac{C_1}{n^{2d}} \sum_{x \in \Zd} L_n^{(1)}\phi(x/n)^2  \, dt, 
\end{align}
where in the last inequality we applied Assumption \ref{HydroAssump} in the same fashion than in the proof of Proposition \ref{Growthcontroltwo}. Notice that by Mosco convergence of two and one particles the last expression is bounded, and hence
 \begin{equation}\label{FirstAldous} 
     \lim_{A \to \infty} \limsup_{n \to \infty} \P_{\mu_n}\left(\sup_{t \in [0,T]} \left|  x_t^{(n)} \right| > A \right) =0.
     \end{equation} 

\noindent
Proposition 1.6 of \cite{kipnis1998scaling} is shown analogously obtaining the estimate
\begin{align}
    \E_n &\left[  \left( \int_\tau^{\tau+\lambda} Y_s^{(n)}(L_n^{(1)} \phi,\eta) \, ds \)^ 2  \right] \nn \\
    &\leq \delta \int_\tau^{\tau+\lambda} \E_n \left[  \left(  Y_s^{(n)}(L_n^{(1)} \phi,\eta) \)^2  \right] \, ds \nn \\
    &\leq \delta \int_\tau^{\tau+\lambda}  \frac{\alpha (\alpha+1)}{n^d} \sum_{x \in \Zd} \sum_{x\myprime, y\myprime \in \Zd} p_{r(n)t}^{(2)}(x/n,x/n; x\myprime/n, y\myprime/n) \rho_0(x\myprime/n) \rho_0(y\myprime/n)  L_n^{(1)}\phi(x/n)^2 \, dt \nn \\
    &+ \delta \int_\tau^{\tau+\lambda}  \frac{\alpha}{n^d} \sum_{x \in \Zd} \sum_{x\myprime \in \Zd} p_{r(n)t}^{(1)}(x/n,x\myprime/n) \rho_0(x\myprime/n) L_n^{(1)}\phi(x/n)^2  \, dt \nn \\
    &+ \delta \int_\tau^{\tau+\lambda}  \frac{\alpha (\alpha+1)}{n^{2d}} \sum_{x \in \Zd}   L_n^{(1)}\phi(x/n)^2 \, dt + T \int_0^T  \frac{\alpha}{n^{2d}} \sum_{x \in \Zd} L_n^{(1)}\phi(x/n)^2 \, dt.  
\end{align}
This finishes the proof of the tightness for the integral term.

\appendix
\section{Appendix}

\subsection{A class of jump processes on $\Rd$}
Here we provide two properties of the process $X_t$ with state space $\Rd$ associated with the generator
$L$. Namely, the regularity of its Dirichlet form and the ultracontractivity of its semi-group. The first
property allows us to talk about convergence of Dirichlet forms, while the second one is a technical
requirement to show uniqueness of solutions to the limiting SPDE governing the fluctuations.

\subsubsection*{Regularity of Dirichlet form}
The process $X_t$ with generator $L$ given by~\eqref{DefL} can also be studied in terms of its Dirichlet form
$(\caE, D(\caE))$. This form is given by
\begin{align}
  \label{DefDirFormMacro01}
  \caE(f) &= \int_{\Rd} \int_{\Rd} q(y) (f(x+y)-f(x))^2 dy \, dx, \nn 
\end{align} 
with
\begin{equation*}
  D(\caE)= \left\{ \, f \in L^2(\R^{d}, dx) \;\middle|\; \caE(f) < \infty  \, \right\}.
\end{equation*} 

\begin{remark}
  \label{RemDirFormOne}
  It is known, see for example~\cite{fukushima2010dirichlet,schilling2012structure}, that if the functions
  $\Psi_1 \colon \R^{d} \to \R$ and $\Psi_2 \colon \R^{d} \to \R$, given by
  \begin{equation}
    \Psi_1(x) := \int_{ 0< |x-y|\leq 1 } \left|x-y \right| \, q(y) \, dy  \nn 
  \end{equation} 
  and
  \begin{equation}
    \Psi_2(x) := \int_{ |x-y|>1 }  q(y) \, dy,  \nn 
  \end{equation} 
  are elements of $L_{\text{loc}}^1(\Rd)$, then $(\caE,C_k^\infty(\R^{d}))$ is a closable Markovian symmetric
  form on $L^2(\Rd)$, and the closure $(\caE,\caF)$ is a regular symmetric Dirichlet form. Notice that this is
  in particular true for the functions $q$ considered in this paper.
\end{remark}

\subsubsection*{Ultracontractivity}
It is known (see~\cite{coulhon1996ultracontractivity},~\cite{barlow2009heat}, and~\cite{hu2006nash}) that
under Assumption \ref{AssumpNashIneq} the semigroup $S_t^L$, associated to $\caE$, satisfies
\begin{equation}\label{Ultracontractivitity} 
\norm{S_t^L f}_\infty \leq C \norm{f}_1,
\end{equation} 
for all $f \in L^1(\Rd)$.

\subsection{Mosco convergence of $\SIP$ particles}
\label{sec:app-Mosco}
In this section we state some consequences of Assumption~\ref{RemarkConseqInvariance} in the context of
Dirichlet forms and Mosco convergence. In particular, using the results of~\cite{ayala2024mosco}, we show
the convergence of the Dirichlet form of $k$-$\SIP$ particles to the Dirichlet form associated to the process
given by $k$-independent copies of the process with generator $L$.

\subsubsection{Convergence of Hilbert spaces}
\label{sec:app-Hilbert}
The convergence of Dirichlet forms we are interested in takes place in the setting of convergence of a
sequence of Hilbert spaces. We recall this notion of convergence introduced in~\cite{kuwae2003convergence}.

\begin{definition}[Convergence of Hilbert spaces]
  \label{HilConv}
  A sequence of Hilbert spaces $\{ H_n \}_{n \geq 0}$ \emph{converges} to a Hilbert space $H$ if a dense
  subset $\Cfrak \subseteq H$ exists and a family of linear maps
  $\left\{ \Phi_n \colon \Cfrak \to H_n \right\}_n$ such that
  \begin{equation}
    \label{HilCond}
    \lim_{n \to \infty} \| \Phi_n f \|_{H_n} = \| f \|_{H}, \qquad \text{  for all } f \in \Cfrak. \nn 
  \end{equation}
\end{definition}

It is also necessary to introduce the concepts of strong and weak convergence of vectors defined on a
convergent sequence of Hilbert spaces. In 
Definitions~\ref{strongcon}, \ref{weakcon} and~\ref{MoscoDef} we assume that the spaces $\{ H_n \}_{n \geq 0}$
converge to the space $H$, in the sense just defined, with the dense set $\Cfrak \subset H$ and the sequence
of operators $\{ \Phi_n \colon \Cfrak \to H_n \}_n$ witnessing the convergence.

\begin{definition}[Strong convergence on Hilbert spaces]
  \label{strongcon}
  A sequence of vectors $\{ f_n \}$ with $f_n$ in $H_n$, is said to \emph{strongly converge} to a vector
  $f \in H$ if there exists a sequence $\{ \tilde{f}_M \} \in \Cfrak$ such that \be \lim_{M\to \infty} \|
  \tilde{f}_M -f \|_{H} = 0 \nn
\end{equation}
and
\begin{equation}
  \lim_{M \to \infty} \limsup_{n \to \infty} \| \Phi_n \tilde{f}_M -f_n \|_{H_n} = 0. \nn 
\end{equation}
\end{definition}

\begin{definition}[Weak convergence on Hilbert spaces]\label{weakcon}
  A sequence of vectors $\{ f_n \}$ with $f_n \in H_n$, is said to \emph{weakly converge} to a vector $f$ in a
  Hilbert space $H$ if
  \begin{equation}
    \lim_{n\to \infty} \left \langle f_n, g_n \right \rangle_{H_n} = \ \left \langle f,
      g \right \rangle_{H}, \nn
  \end{equation}
  for every sequence $\{g_n \}$ which converges strongly to $g \in H$.
\end{definition}

\begin{remark}
  Notice that, as expected, strong convergence implies weak convergence, and, for any $f \in \Cfrak$, the
  sequence $\Phi_n f $ strongly-converges to $f$.
\end{remark}

\noindent
Given these notions of convergence, we can also introduce related notions of convergence for operators.
We denote by $L(H)$ the set of all bounded linear operators in $H$.

\begin{definition}[Convergence of bounded operators on Hilbert spaces]
  \label{opercon}
  A sequence of bounded operators $\{ T_n \}$ with $T_n \in L(H_n)$  \emph{converges strongly
      (resp. weakly)} to an operator $T$ in $L(H)$ if for every strongly (resp. weakly) convergent sequence
  $\{ f_n \}$ with $f_n \in H_n$ and limit $f \in H$, the sequence $\{ T_n f_n \}$ converges
  strongly (resp. weakly) to $T f$.
\end{definition}

\noindent
We are now ready to introduce Mosco convergence.

\subsubsection{Definition of Mosco convergence}
In this section we assume the Hilbert convergence of a sequence of Hilbert spaces $\{ H_n \}_n$ to $H$.

\begin{definition}[Mosco convergence]
  \label{MoscoDef}
  A sequence of Dirichlet forms $\{ (\caE_n, D(\caE_n))\}_n $, defined on Hilbert spaces $H_n$, \emph{Mosco
    converges} to a Dirichlet form $(\caE, D(\caE)) $, defined on some Hilbert space $H$, if:
  \begin{description}
  \item[Mosco I.] For every sequence of $f_n \in H_n$ weakly converging  to $f$ in $H$, one has
    \begin{equation}\label{mosco1}
      \caE ( f ) \leq \liminf_{n \to \infty} \caE_n ( f_n ).
    \end{equation}
  \item[Mosco II.] For every $f \in H$, there exists a sequence $ f_n \in H_n$ strongly converging to $f$ in
    $H$, such that
    \begin{equation}\label{mosco2}
      \caE ( f) = \lim_{n \to \infty} \caE_n ( f_n ).
    \end{equation}
  \end{description}
\end{definition}

The Markovian properties of the Dirichlet form correspond to the properties of the associated semigroups and
resolvents. The following theorem from~\cite{kuwae2003convergence}, which relates Mosco convergence with
convergence of semigroups and resolvents, is a powerful application of this correspondence and one of the main
ingredients of our work.

\begin{theorem}
  \label{MKS}
  Let $\{ (\caE_n, D(\caE_n))\}_n $ be a sequence of Dirichlet forms on Hilbert spaces $H_n$ and let
  $(\caE, D(\caE)) $ be a Dirichlet form in some Hilbert space $H$. The following statements are equivalent:
  \ben
\item $\{ (\caE_n, D(\caE_n))\}_n $ Mosco-converges to $\{ (\caE, D(\caE))\} $.
\item The associated sequence of semigroups $\{ S_{n} (t) \}_n $ strongly converges to the semigroup $ S(t)$
  for every $t >0$.  \een
\end{theorem}

\subsubsection{Some comments on Mosco convergence and Assumption~\ref{RemarkConseqInvariance}}
\label{MoscoConvandAssumptionInvariance}
In this section we explain in more detail the content of Assumption~\ref{RemarkConseqInvariance}. From the
definitions of the previous sections we can see that Assumption~\ref{RemarkConseqInvariance}
ensures the following statements:
\begin{itemize}
\item Convergence of Hilbert spaces with a particular projection operator $\Phi_n$ given by the restriction to
  $\tfrac{1}{n}\Zd$
\item  That condition Mosco I is satisfied for any weakly convergent sequence.
\end{itemize}

The first item is easy to prove. We refer to~\cite{ayala2021condensation} for a proof of this fact which does
not depend on the size of the jumps of the underlying process.
The second point is rather restrictive, but we believe that can be shown along the same lines as in \cite{ayala2021condensation} for independent particles in the finite range situation. Namely, it is enough to have that for all $g \in H$, there exists a sequence $ g_n \in H_n $  strongly-converging to $g$ such that
\begin{equation}\label{dualcondual}
\lim_{n\to \infty}\caE^*_n(g_n)= \caE^*(g),
\ee
where $*$ denotes the conjugate dual.\\

Besides $\Phi_n \phi$ converging strongly to $\phi$ we also have the following convergence result consequence
of Assumption~\ref{RemarkConseqInvariance}.

\begin{proposition}
  \label{StrongConvLn1Gamman1}
  For each $\phi \in C_k^\infty(\R^{d})$, the following strong converge 
  with respect to $\SIP$-Hilbert convergence:
  \begin{enumerate}
  \item The sequence $\lbrace L_n^{(1)} \Phi_n \phi : n \in \N \rbrace$ converges strongly to the function
    $L \phi$.
  \item The sequence $\lbrace \Gamma_n^{(1)} \Phi_n \phi : n \in \N \rbrace$ converges strongly to the
    function $\Gamma \phi$.
  \end{enumerate}
\end{proposition} 

\begin{proof} 
  We only show the first statement. The second one is analogous. Let $\phi \in C_k^\infty(\Rd)$. To show the
  strong convergence $ L_n^{(1)} \Phi_n \phi$ to $L \phi$, it is enough to find a sequence of functions
  $\tilde{f}_M \in C_k^\infty(\Rd)$ satisfying
  \begin{equation}
    \label{StrongCondone}
    \lim_{M\to \infty} \| \tilde{f}_M -L \phi \|_{H} = 0
  \end{equation}
  and
  \begin{equation}
    \label{StrongCondtwo}
    \lim_{M \to \infty} \limsup_{n \to \infty} \| \Phi_n \tilde{f}_M - L_n^{(1)} \Phi_n \phi \|_{H_n} = 0. 
  \end{equation}
  Since $q \in L^2(\Rd)$, there exists a sequence $q_M \in C_k^\infty(\Rd)$ such that
  \begin{equation}
    \label{qmL2toq} 
    \lim_{M \to \infty} \norm{q_M -q}_H = 0.
  \end{equation} 
  We can then define the sequence $\tilde{f}_M \in C_k^\infty(\Rd)$ as
  \begin{equation}
    \tilde{f}_M := \phi * q_M  -\phi. \nn 
  \end{equation} 
  Notice that then
\begin{equation}
 \| \tilde{f}_M -L \phi \|_{H} = \norm{\phi * q_M  - \phi*q }_H \leq \norm{\phi}_1 \norm{q_M - q}_H, \nn 
\end{equation}
where in the last line we used Young's convolution inequality. From \eqref{qmL2toq} we obtain \eqref{StrongCondone}.\\

\noindent
For the second condition~\eqref{StrongCondtwo} we estimate 
\begin{align}\label{splittingnormStrongConv}
    \norm{\Phi_n \tilde{f}_M -  L_n^{(1)} \Phi_n \phi}_{H_n}^2 &=  \norm{\Phi_n \phi *q_M - \Phi_n \phi - L_n^{(1)} \Phi_n \phi}_{H_n}^2 \nn \\
    &\leq \norm{\Phi_n \phi *q_M - R_n \Phi_n \phi}_{H_n}^2 +\norm{\Phi_n \phi - S_n \Phi_n \phi}_{H_n}^2,
\end{align}
where
\begin{equation}
  R_n \Phi_n \phi(x/n) = \frac{1}{n^d} \sum_{y \in \Zd} q((y-x)/n) \phi(y/n) \nn 
\end{equation} 
and 
\begin{equation}
  S_n \Phi_n \phi(x/n) = \phi(x/n) \left(\frac{1}{n^d} \sum_{y \in \Zd} q((y-x)/n)\right). \nn 
\end{equation} 
We first deal with the second term in the right hand side of \eqref{splittingnormStrongConv}. Namely, 
\begin{align}
  \norm{\Phi_n \phi - S_n \Phi_n \phi}_{H_n}^2
  &= \frac{1}{n^d} \sum_{x \in \Zd} \phi^2(x/n) + \frac{1}{n^{3d}}
    \sum_{x, y_1, y_2 \in \Zd} \phi^2(x/n) q((y_1-x)/n) q((y_2-x)/n) \nn \\
  &\quad{}- \frac{2}{n^{2d}} \sum_{x, y_1 \in \Zd} \phi^2(x/n) q((y_1-x)/n), \nn 
\end{align}
which by~\eqref{massoneforq} vanishes as $n \to \infty$.

In the same way we also estimate the first term in the right hand side
of~\eqref{splittingnormStrongConv}. Namely,
\begin{align}
  \norm{\Phi_n \phi *q_M - R_n \Phi_n \phi}_{H_n}^2
  &= \frac{1}{n^d} \sum_{x \in \Zd} (\phi*q_M)^2(x/n) \nn \\
  &\quad{}+ \frac{1}{n^{3d}} \sum_{x, y_1, y_2 \in \Zd} \phi(y_1/n) q((y_1-x)/n)  \phi(y_2/n) q((y_2-x)/n) \nn \\
  &\quad{}- \frac{2}{n^{2d}} \sum_{x, y_1 \in \Zd} (\phi*q_M)(x/n) \phi(y_1/n) q((y_1-x)/n). \nn 
\end{align}
As $n \to \infty$, the previous expression converges to
\begin{equation*}
  \int (\phi*q_M)^2(x) \, dx + \int (\phi*q)^2(x) \, dx - 2 \int (\phi*q_M)(x) (\phi*q)(x) \, dx,
\end{equation*}
which vanishes as $M \to \infty$.
\end{proof}

\subsubsection*{Mosco convergence for $k$ particles}
The following result is a consequence of~\cite[Theorem 3.1 and Corollary 5.2]{ayala2024mosco}. By a procedure
completely analogous to the proof of Theorem 3.1 in \cite{ayala2024mosco}, we obtain the convergence of $k$
independent particles in our setting. Moreover, the same proof of Corollary 5.2 in the same reference applies
to our setting, giving us the following result.

\begin{theorem}
  \label{Moscoconvkparticles} 
  The sequence of Dirichlet forms corresponding to $k$-$\SIP$ particles in coordinate notation converges to
  the Dirichlet form associated to the process given by $k$ independent copies of the process with
  infinitesimal generator $L$.
\end{theorem} 

\begin{remark} 
  Notice that Theorem~\ref{Moscoconvkparticles} in particular implies that the sequence of $k$-particle
  semigroups $S_{n}^{(k)}(t) $ converges strongly, with respect to $k$-$\SIP$-Hilbert convergence (or
  $\IRW$-Hilbert convergence), to $k$-times the product of one particle semigroups associated to generator $L$
  given by \eqref{DefL}. Moreover in the same spirit than Proposition \ref{StrongConvLn1Gamman1} we have that
  for any $k$ finite, the sequence
  \begin{equation}
    \prod_{j=1}^k L_n^{(1)} \Phi_n \phi \nn 
  \end{equation} 
  converges strongly, with respect to $k$-$\SIP$-Hilbert convergence, to
  \begin{equation}
    \prod_{j=1}^k L \phi. \nn 
  \end{equation} 
  An analogous statement for the square operators $\Gamma_n$ also holds.
\end{remark}

\subsection{Uniqueness of distributions of the generalized time depedent Ornstein-Uhlenbeck process}
\label{AppendixUniquenessGTOUP}

For the readers convenience we include the proof of Proposition~\ref{PropoUniquenessGTOUP}.
We start with an auxiliary statement.
\begin{proposition} 
  For all $f \in L^1(\Rd)$, the mapping
  \begin{equation}
    t \mapsto \int_0^t \norm{ \sqrt{\Gamma_s^{\rho}f}}^2_{L^2(\Rd, \rho(s,x) \, dx)} \, ds,  \nn 
  \end{equation} 
  is continuous in $[0,T]$.
\end{proposition}

\begin{proof} 
It is enough to notice that by~\eqref{Ultracontractivitity} 
\begin{equation}
  \norm{S_t^L f}_{\infty} \leq C \norm{f}_{1}, \nn 
\end{equation} 
for some $C$ that does not depend on $f$ or $t$.
\end{proof} 

\begin{proofof}{Proposition~\ref{PropoUniquenessGTOUP}}
  By~\cite[Proposition 3.4]{revuz2013continuous} we know that for all $g \colon [0,T] \to S(\Rd)$ the process
  \begin{equation}
    Z_t(g) = \exp \left\lbrace i \left(Y_t(g) - \int_0^t Y_s((\partial_s + L) g) \, ds \right) + \int_0^t \norm{ \sqrt{\Gamma_s^{\rho}g}}^2_{L^2(\Rd, \rho(s,x) \, dx)} \, ds\right\rbrace \nn 
\end{equation} 
is a local Martingale. To make $\{\, Z_t(g) \mid t \in [0,T] \, \}$ an actual martingale, it is enough to have
\begin{equation}
  \E \left[ \sup_{ t \in [0,T]} \left|  Z_t(g) \right|  \right] < \infty, \nn 
\end{equation} 
which is a consequence of the mapping
\begin{equation}
  t \mapsto \int_0^t \norm{ \sqrt{\Gamma_s^{\rho}g}}^2_{L^2(\Rd, \rho(s,x) \, dx)} \, ds \nn 
\end{equation} 
being continuous. Hence $\{\, Z_t(g) \mid t \in [0,T] \, \}$ is a martingale. In particular, if we plug
$g(t,x)= S_{T-t}^L \phi(x)$ (where $S^L$ is the semi-group associated to $L$), we obtain that
\begin{equation}
  \bar{Z}_t(\phi) =  \exp \left\lbrace i Y_t(S_{T-t}^L \phi) + \int_0^t
    \norm{ \sqrt{\Gamma_s^{\rho} S_{T-s}^L \phi}}^2_{L^2(\Rd, \rho(s,x) \, dx)} \, ds \right\rbrace \nn 
\end{equation} 
is also a martingale. From here, after substitution of $\phi=\lambda \phi$ and $T=t$, we find that for
$s \leq t$
\begin{equation}
  \label{CharacYt} 
  \E \left[ \exp \left\lbrace i \lambda Y_t(\phi) \right\rbrace   \mid \caF_s \right]
  = \exp \left\lbrace i \lambda Y_s(S_{t-s}^L \phi) - \lambda^2 \int_s^t
    \norm{ \sqrt{\Gamma_u^{\rho} S_{t-u}^L \phi}}^2_{L^2(\Rd, \rho(u,x) \, dx)} \, ds \right\rbrace.
\end{equation} 
This implies that conditioning on $\caF_s$, $Y_t(\phi)$ is Gaussian with mean
\begin{equation*}
  Y_s(S_{t-s}^L \phi)
\end{equation*}
and variance 
\begin{equation*}
  2 \int_s^t \norm{ \sqrt{\Gamma_u^{\rho} S_{t-u}^L \phi}}^2_{L^2(\Rd, \rho(u,x) \, dx)} \, ds .
\end{equation*}
Then a standard Markov argument (see for example~\cite[Chapter 11]{kipnis1998scaling} or the explicit
computations in~\cite[Section 5.3]{erhard2020non}) guarantees the uniqueness of finite-dimensional
distributions.
\end{proofof} 




\bibliography{biblio} 
\bibliographystyle{plain} 
\end{document}